\documentclass{ws-ccm}
\usepackage{amssymb}

\def\dint{\displaystyle\int}

\def\non{\nonumber}

\begin{document}

\title{Fast reaction limit of reaction diffusion systems with nonlinear diffusion}

\author{Elaine Crooks}
\address{Department of Mathematics, Faculty of Science and Engineering,\\ Swansea University, Swansea SA1 8EN, UK, e.c.m.crooks@swansea.ac.uk}

\author{Yini Du}
\address{Department of Mathematics, Faculty of Science and Engineering,\\ Swansea University, Swansea SA1 8EN, UK, yini.du@swansea.ac.uk}

\maketitle

\begin{abstract}
In this paper, we present an approach to characterising fast-reaction limits of systems with nonlinear diffusion, when there are either two reaction-diffusion equations, or one reaction-diffusion equation and one ordinary differential equation, on unbounded domains. Here, we replace the terms of the form $u_{xx}$ in usual reaction-diffusion equation, which represent linear diffusion, by terms of form $\phi(u)_{xx}$, representing nonlinear diffusion. We prove the convergence as $k\rightarrow\infty$ to a limit that is determined by the unique solution of a certain scalar nonlinear diffusion limit problem.
\end{abstract}


\section{Introduction}
 Nonlinear diffusion is needed in certain modelling scenarios to describe processes involving fluid flow, heat transfer or diffusion. For instance, it can describe the flow of an isentropic gas through a porous medium \cite{gas}. In this paper, we study the reaction diffusion systems with nonlinear diffusion in one-dimensional spatial domains.
A prototype for the form of nonlinear diffusion considered with this paper is $(u^m)_{xx}$, where $m>1$. We will consider general nonlinear diffusion terms of the form $\phi(u)_{xx}$, where the function $\phi\in C^2(\mathbb{R})$, $\phi$ and $\phi'$ are assumed to be strictly increasing with
\begin{align}
\phi(s)>0\ {\rm as}\ s>0\ {\rm and}\ \phi'(s)=\phi(s)=0\ {\rm when}\ s=0. \label{phi}
\end{align}

The problems we consider are motivated by a prototype reaction diffusion system which consists of a chemical A and an immobile substrate B that react in a semi-infinite region. We denote by $u$ the concentration of A and $v$ the concentration of B at $x\in(0,\infty)$ and time $t\in(0,T)$. The reaction of $u$ and $v$ is described as
\begin{align}
\left\{
\begin{array}{llll}
  u_{t}=u_{xx}-kuv,\\
  v_{t}=-kuv,
 \end{array}
\right.\label{m}
\end{align}
 where $k$ is the rate constant of the reaction (which is positive). The chemical reaction can be modelled for simplicity by the one-dimensional spatial domain $(0,\infty)$ with $u=U_0$ at the surface $x=0$. That $u$ and $v$ are nonnegative is natural since they typically correspond to concentration of chemical substances.

 In \cite{1996}, Hilhorst, van der Hout and Peletier study the asymptotic behaviour of $k$-dependent solutions $(u^k,v^k)$ of (\ref{m}) as $k\rightarrow\infty$ (i.e. the reaction is very fast). They establish a free boundary problem which is satisfied in the limit when solutions $(u^k,v^k)$ converge to a self-similar limit $(u,v)\left(\frac{x}{\sqrt t}\right)$ as $k\rightarrow\infty$. The free boundary has the form $x=a\sqrt t$, where $a>0$ and divides the area in which the mobile chemical A is present from the area where A is absent. The fast-reaction limit of (\ref{m}) can be motivated by the study of penetration of radio-labeled antibodies into tumourous tissue since the attachment of antibodies to antigens in the tissue may react very fast.

Modelling can give rise to other systems related to (\ref{m}) such that the fast-reaction limit in which one mobile substance invades a \textit{mobile} substrate. Among other problems, Crooks and Hilhorst \cite{selfsim} study the system analogous to (\ref{m}) when reactant $u$ and substrate $v$ are both mobile, for example, when carbonic acid penetrates into water. In this case, the substrate will diffuse, which is modelled by introducing a term $d_vv_{xx}$ where $d_v>0$. The paper \cite{selfsim} is concerned with the free boundary problems in the limit that $k\rightarrow\infty$ in four cases: $d_v>0$ with two mobile reactants, $d_v=0$ with one mobile and one immobile reactant, problems defined on the spatial domain $(0,\infty)$ as in (\ref{a}) and also on the whole real line $\mathbb{R}$, which can arise, for instance, in modelling neutralisation of an acid and a base that are initially separated. In all four cases, the free boundary has the form $x=a\sqrt t$ where the constant $a$ is determined by a different equation in each case and plays an important role in characterising the rate of penetration of one substance into the other in the limit $k\rightarrow\infty$. When the problem is considered on the spatial domain $\mathbb{R}$ with $d_v>0$, the constant $a$ in the corresponding limit is not necessarily positive. Note that when $a>0$, substance $u$ penetrates into substance $v$, while on the other hand, $v$ penetrates into $u$ when $a<0$. For each of the problems with $d_v\ge0$ on both the spatial domains $\mathbb{R}$ and $(0,\infty)$, an explicit formula is given in \cite{selfsim} for the self-similar limit function.

An analogue of (\ref{m}) with nonlinear diffusion in bounded multi-dimensional domains is studied in \cite{nonli} by Hilhorst, van der Hout and Peletier. They consider the substrate $u$ with nonlinear diffusion modelled with a term $\Delta\phi(u)$, where $\phi(u)=\int_0^uD(s){\rm d}s$ and $D$ is the diffusivity of the medium. Under assumptions in \cite{nonli}, $D(s)$ may vanish at $s=0$, so the equation for $u$ need not be uniformly parabolic. Thus \cite{nonli} focuses on weak solutions since it is possible that the system studied has no classical solution. In studying of the multi-dimensional limiting free boundary problems in \cite{nonli}, the free boundary $\Gamma(t)$ of the limit problem is assumed as a smooth surface that lies entirely within the bounded domain and varies smoothly with $t$.

We treat two pairs of problem with nonlinear diffusion terms on the spatial domains $\mathbb{R^+}$ and $\mathbb{R}$. The first pair of problems defined on the half-strip $S_{T}:=\left\{(x,t):0<x<\infty,0<t<T\right\}$, one with $\varepsilon>0$ and the other with $\varepsilon=0$, are
\begin{align}
\left\{
\begin{array}{llll}
  u_{t}=\phi(u)_{xx}-kuv,\quad &(x,t)\in(0,\infty)\times(0,T),\\
  v_{t}=\varepsilon\phi(v)_{xx}-kuv,\quad &(x,t)\in(0,\infty)\times(0,T),\\
  u(0,t)=U_{0},\quad \varepsilon \phi(v)_{x}(0,t)=0,\quad &\mbox{for}\quad t\in(0,T),\\
  u(x,0)=u^k_0(x),\quad v(x,0)=v^k_{0}(x),\quad &\mbox{for}\quad x\in\mathbb{R}^{+}.
\label{a}
\end{array}
\right.
\end{align}
 As in \cite{1996} $kuv$ is the contribution of a chemical reaction where $k$ determines the reaction rate. We define, as in \cite{selfsim}, the initial data for the limiting self-similar solutions as
\begin{align*}
u_0^\infty=\left\{\begin{array}{llll}
&U_0\quad &x=0,\\ &0\quad &x>0,\end{array}\right.\quad
v_0^\infty=\left\{\begin{array}{llll}
&0\quad &x=0,\\ &V_0\quad &x>0,\end{array}\right.
\end{align*}
which equal constant initial conditions on the half-line in \cite{1996}, where $U_0$ and $V_0$ are positive constants, and choose the initial data $u_0^k,v_0^k$ that satisfy
\begin{itemize}
\item[{\rm (i)}]$u_0^k,v_0^k\in C^2(\mathbb{R^+})$;
\item[{\rm (ii)}]$0\le u_0^k\le U_0,\ 0 \le v_0^k\le V_0$;
\item[{\rm (iii)}]$u_0^k\rightarrow u_0^\infty,\ v_0^k\rightarrow v_0^\infty$ in $L^1(\mathbb{R^+})$ as $k\rightarrow\infty$.
\item[{\rm (iv)}]For each $r>0$, there exists a continuous function $\omega_r:\mathbb{R}^+\mapsto\mathbb{R}^+$ with $\omega_r(\mu)\rightarrow 0$ as $\mu\rightarrow 0$ and
\begin{align*}
\Vert u^k_0(\cdot+\delta)-u_0^k(\cdot)\Vert_{L^1((r,\infty))}+\Vert v^k_0(\cdot+\delta)-v_0^k(\cdot)\Vert_{L^1((r,\infty))}\le\omega_r(\delta),
\end{align*}
for all $k>0$, $\delta<\frac{r}{4}$.
\end{itemize}

For both $\varepsilon=0$ and $\varepsilon>0$, we will prove the existence and uniqueness of weak solutions $(u^k,v^k)$ of problem (\ref{a}) for every $k>0$, and study the asymptotic behaviour of $(u^k,v^k)$ as $k\rightarrow\infty$. As we will see, the limits $u$ of $u^k$ and $v$ of $v^k$ are separated by a free boundary and given by the positive and negative parts respectively of a function $w$, where $w$ satisfies the limit problem (\ref{c}) and
\begin{align*}
u=w^+ \ {\rm and}\ v=-w^-,
\end{align*}
where $s^+=\max\{0,s\}$ and $s^-=\min\{0,s\}$. The $k\rightarrow\infty$ limit problem (\ref{c}) is a scalar problem, where the nonlinear diffusion function (\ref{c0}) depends on whether $w$ is positive or negative. We prove that there exists a unique weak solution of the limit problem (\ref{c}) in Theorem \ref{tu3}.

The second pair of problems is defined on the strip $Q_{T}:=\left\{(x,t):x\in\mathbb{R}, 0<t<T\right\}$, one with $\varepsilon>0$ and the other one with $\varepsilon=0$ are
\begin{align}
\left\{
\begin{array}{llll}
  u_{t}=\phi(u)_{xx}-kuv,\quad &(x,t)\in\mathbb{R}\times(0,T),\\
  v_{t}=\varepsilon\phi(v)_{xx}-kuv,\quad &(x,t)\in\mathbb{R}\times(0,T),\\
  u(x,0)=u_0^k(x),\quad v(x,0)=v_0^k(x),\quad &\mbox{for}\quad x\in\mathbb{R},
\label{a2}
\end{array}
\right.
\end{align}
where we define, as in \cite{selfsim} that
\begin{align*}
u_0^\infty=\left\{\begin{array}{llll}
&U_0\quad &x<0,\\ &0\quad &x>0,\end{array}\right.\quad
v_0^\infty=\left\{\begin{array}{llll}
&0\quad &x<0,\\ &V_0\quad &x>0,\end{array}\right.
\end{align*}
with $U_0, V_0$ positive constants, $k$ as in (\ref{a}) and initial data $u_0^k,v_0^k$ satisfy
\begin{itemize}
\item[{\rm (i)}]$u_0^k,v_0^k\in C^2(\mathbb{R})$;
\item[{\rm (ii)}]$0\le u_0^k\le U_0,\ 0 \le v_0^k\le V_0$;
\item[{\rm (iii)}]$u_0^k\rightarrow u_0^\infty,\ v_0^k\rightarrow v_0^\infty$ in $L^1(\mathbb{R})$ as $k\rightarrow\infty$.
\item[{\rm (iv)}]There exists a continuous function $\omega:\mathbb{R}^+\mapsto\mathbb{R}^+$ with $\omega(\mu)\rightarrow 0$ as $\mu\rightarrow 0$ and
\begin{align*}
\Vert u^k_0(\cdot+\delta)-u_0^k(\cdot)\Vert_{L^1(\mathbb{R})}+\Vert v^k_0(\cdot+\delta)-v_0^k(\cdot)\Vert_{L^1(\mathbb{R})}\le\omega(\delta),
\end{align*}
for all $k>0$, $\delta\in\mathbb{R}$.
\end{itemize}

Note that for simplicity, we use the same notation $u^\infty_0,v^\infty_0$ for both half-line and whole line initial functions. We again consider both the case of two mobile reactants where $\varepsilon>0$, and the case of one mobile and one immobile reactant, when $\varepsilon=0$. Similarly to the half-line case, we prove the existence and uniqueness of weak solutions $(u^k,v^k)$ of problem (\ref{a2}), and study the convergence to self-similar limit profiles $(u,v)$ as $k\rightarrow\infty$, where $u$ and $v$ are given by a function $w$, the unique weak solution of the limit problem (\ref{c2}).

The work of this paper continues and extends earlier studies of fast-reaction limits \cite{selfsim}\cite{1996}\cite{nonli}, by introducing the nonlinear function $\phi$, in both the case of two mobile reactants $(\varepsilon>0)$ in addition to that of one mobile reactant $(\varepsilon=0)$ and in considering the whole-line problem (\ref{a}) in addition to the half-line problem (\ref{a2}). In \cite{nonli}, the existence of weak solutions is proved by looking at a sequence of uniformly parabolic problems in which $\phi'_n(u)\ge\frac{1}{n}$ and studying the solutions in the limit as $n\rightarrow\infty$. We exploit some ideas and an iterative method from \cite{nonli}, but our domains are unbounded and when $\varepsilon>0$, the equations for both $u$ and $v$ of (\ref{a}) and (\ref{a2}) have nonlinear diffusion and are not uniformly parabolic. In the problems treated in \cite{selfsim}, where the diffusion is linear and the problems are studied in unbounded domains, the overall strategy and a series of cut-off functions are useful in studying the $k\rightarrow\infty$ limit. Here, we consider $\phi(u^k),\phi(v^k)$ rather than $u^k,v^k$ and in order to deal with the nonlinear diffusion, alternative methods and additional procedures are needed, for example, in proving the estimates of the differences of time translate, there will be an extra term because of the nonlinear diffusion.

This paper is organised as follows. In Section 2, we study the half-line problem (\ref{a}), starting with the uniqueness of weak solutions for (\ref{a}). Under the assumptions on $\phi$ in (\ref{phi}), the equations for $u,v$ need not be uniformly parabolic when $\varepsilon>0$, so the existence of weak solution for (\ref{a}) are proved in Theorem \ref{er} by an iterative method. Section 3 is concerned with passing to the limit as $k\rightarrow\infty$ of the weak solutions $(u^k,v^k)$, via some \textit{a priori} estimates and a key bound on $ku^kv^k$ in $L^1(S_T)$, independent of $k$ and $\varepsilon\ge0$ which is proved in Theorem \ref{tu}. Section 5 contains the whole-line counterparts of the study of the half-line problem in Sections 2-3.

The unique weak solutions of the limit problems (\ref{c}) and (\ref{c2}) can in fact be shown to be self-similar solutions, as was established in \cite{selfsim} for the case of linear diffusion. Note that in contrast, to the case of linear diffusion \cite{selfsim}, we know of no explicit self-similar solutions for the $k\rightarrow\infty$ limit
problems with nonlinear diffusion that are obtained here. We will present the study of the self-similar solutions of the limit problem elsewhere \cite{selfsimnon}. It can be shown that the limit function $w$ of (\ref{c}) satisfies one of two self-similar problems, depending on whether $\varepsilon>0$ or $\varepsilon=0$. The function $f:\mathbb{R}^+\rightarrow \mathbb{R}$ describes a self-similar limit solution such that $w(x,t)=f(\eta)$ where $\eta=x/\sqrt{t}$ for $(x,t)\in S_T$. The existence of self-similar solutions of the limit problems is proved by using one or two parameter shooting methods.

\section{Half-line case: existence and uniqueness of weak solutions for $\varepsilon > 0$}
Let $\varepsilon>0$. We consider first an approximate problem to (\ref{a}). Given $R>1$, consider the problem
\begin{align}
\left\{
\begin{array}{llll}
  u_{t}=\phi(u)_{xx}-kuv,\quad &{\rm in}\ (0,R)\times(0,T),\\
  v_{t}=\varepsilon\phi(v)_{xx}-kuv,\quad &{\rm in}\ (0,R)\times(0,T),\\
  u(0,t)=U_{0},\quad  \phi(v)_{x}(0,t)=0,\quad &\mbox{for}\quad t\in(0,T),\\
  \phi(u)_x(R,t)=0,\quad  \phi(v)_{x}(R,t)=0,\quad &\mbox{for}\quad t\in(0,T),\\
  u(x,0)=u_{0,R},\quad v(x,0)=v_{0,R},\quad &\mbox{for}\quad x\in(0,R),
\label{f}
\end{array}
\right.
\end{align}
where $u_{0,R},v_{0,R}\in C^2(\mathbb{R}^+)$ are such that $0\le u_{0,R}\le U_0$, $0\le v_{0,R}\le V_0$ and
\begin{align}
u_{0,R}=u_0^k\beta^R,\qquad v_{0,R}=-(V_0-v_0^k)\beta^R+V_0,\label{ic1}
\end{align}
where the family of cut-off functions $\beta^R\in C^\infty(\mathbb{R}^+)$ with $R>1$ are defined as
\begin{align*}
\beta^{R}=\left\{\begin{aligned}
&1\quad &x\le R-1,\\
  &\beta^{1}(x+2-R)\quad &x\ge R-1.\end{aligned}\right.
\end{align*}
with $\beta^1\in C^{\infty}\left(\mathbb{R}^{+}\right)$ is a non-negative cut-off function such that $0\le\beta^1(x)\le1$ for all $x\in\mathbb{R}^+$, $\beta^1(x)=1$ when $x\le1$ and $\beta^1(x)=0$ when $x\ge2$.

Since $\phi'(s)$ may vanish at $s=0$, the equations for $u$ and $v$ as $\varepsilon>0$ in problem (\ref{f}) need not be uniformly parabolic and it is possible that there is no classical solution. Thus we are led to introduce a notion of a weak solution.

Now define
\begin{align}
\Omega_R:=\left\{\alpha\in W^{1,2}(0,R)|\ \alpha=0 \ {\rm at}\ x=0\right\},
\end{align}
and let $\hat u\in C^\infty(\mathbb {R}^+)$ be a smooth function that $\hat u=U_0$ when $x=0$ and $\hat u=0$ when $x>1$.
\begin{definition}
A pair $(u_R,v_R)\in L^\infty\left((0,R)\times(0,T)\right)\times L^\infty\left((0,R)\times(0,T)\right)$ is called a weak solution of (\ref{f}) if
{\rm(i)} $\phi(u_R)\in\phi(\hat u)+L^2(0,T;\Omega_R)$,\ $\phi(v_R)\in L^2(0,T;W^{1,2}(0,R))$;\\
{\rm(ii)} $(u_R,v_R)$ satisfies
\begin{align*}
\int_0^Ru_{0,R}\xi(x,0){\rm d}x+\int_0^T\int_0^Ru_R\xi_t{\rm d}x{\rm d}t=\int_0^T\int_0^R\phi(u_R)_x\xi_x{\rm d}x{\rm d}t+k\int_0^T\int_0^R\xi u_Rv_R{\rm d}x{\rm d}t,\\
\int_0^Rv_{0,R}\xi(x,0){\rm d}x+\int_0^T\int_0^Rv_R\xi_t{\rm d}x{\rm d}t=\int_0^T\int_0^R\varepsilon\phi(v_R)_x\xi_x{\rm d}x{\rm d}t+k\int_0^T\int_0^R\xi u_Rv_R{\rm d}x{\rm d}t,
\end{align*}
where $\xi\in \mathcal{F}_T^R:=\left\{\xi\in C^{1}\left([0,R]\times[0,T]\right)|\ \xi(0,t)=\xi(\cdot,T)=0\ {\rm for}\ t\in(0,T)\right\}$.
\end{definition}

We use the following comparison theorem for (\ref{f}) to prove the uniqueness of the weak solution of (\ref{f}). The proof is similar to that of \cite[Lemma 3.2]{selfsim}. We sketch the points here, focusing on the parts where our problem needs a slightly different argument.
\begin{lemma}\label{lur}
Suppose that $\varepsilon\ge0$ and $(\overline{u_R},\overline{v_R})$, $(\underline{u_R},\underline{v_R})$ be such that
\begin{itemize}\item[{\rm (a)}] $\overline{u_R},\underline{u_R}\in L^\infty((0,R)\times(0,T)]);$
\item[{\rm (b)}]$\phi(\overline{u_R})\in\phi(\overline{u_R}(0,\cdot))+L^2(0,T;\Omega_R)$,\ $\phi(\underline{u_R})\in\phi(\underline{u_R}(0,\cdot))+L^2(0,T;\Omega_R);$
\item[{\rm (c)}]$\overline{u_R}_t,\underline{u_R}_t,\phi(\overline{u_R})_{xx},\phi(\underline{u_R})_{xx}\in L^1((0,R)\times(0,T));$
\item[{\rm (d)}]$\overline{v_R},\underline{v_R}\in L^\infty((0,R)\times(0,T));$
\item[{\rm (e)}]If $\varepsilon>0$, $\phi(\underline{v_R}),\phi(\underline{v_R})\in L^2(0,T;W^{1,2}(0,R))$,\
 $\overline{v_R}_t,\underline{v_R}_t,\phi(\overline{v_R})_{xx},\phi(\underline{v_R})_{xx}\in L^1((0,R)\times(0,T)).$
 \end{itemize}
$(\overline{u_R},\overline{v_R})$, $(\underline{u_R},\underline{v_R})$ satisfy
\begin{align*}
&\overline{u_R}_{t}\ge\phi(\overline{u_R})_{xx}-k\overline{u_Rv_R},\quad\underline{u_R}_{t}\le\phi(\underline {u_R})_{xx}-k\underline{u_R}\underline{v_R},&&{\rm in}\ (0,R)\times(0,T),\\
&\overline{v_R}_{t}\le\varepsilon\phi(\overline{v_R})_{xx}-k\overline{u_Rv_R},\quad\underline{v_R}_{t}\ge\varepsilon\phi(\underline{v_R})_{xx}-k\underline{u_R}\underline{v_R},\ &&{\rm in}\ (0,R)\times(0,T),\\
&\overline{u_R}(0,\cdot)\ge\underline{u_R}(0,\cdot),\quad\phi(\overline{v_R})_x(0,\cdot)\le\phi(\underline{v_R})_x(0,\cdot),\ &&\mbox{on} \ (0,T),\\
&\phi(\overline{u_R})_x(R,\cdot)\ge\phi(\underline{u_R})_x(R,\cdot),\quad\phi(\overline{v_R})_x(R,\cdot)\le\phi(\underline{v_R})_x(R,\cdot),&&\mbox{on} \ (0,T),\\
&\overline{u_R}(\cdot,0)\ge\underline{u_R}(\cdot,0),\quad \overline{v_R}(0,\cdot)\le\underline{v_R}(0,\cdot), &&\mbox{on} \ (0,R).
\end{align*}
Then
\begin{align*}
\overline{u_R}\ge\underline{u_R},\quad\overline{v_R}\le\underline{v_R}\quad\mbox{in} \  (0,R)\times(0,T).
\end{align*}
\end{lemma}

\begin{proof}Take a smooth non-decreasing convex function $m^{+}:\mathbb{R}\rightarrow\mathbb{R}$ with
\begin{align*}
m^{+}\ge0,\ m^{+}(0)=0,\ \left(m^{+}\right)'(0)=0,\ m^{+}(r)\equiv0\ \mbox{for} \ r\le0,\ m^{+}(r)=r-\frac{1}{2},
\end{align*}
for $r>1$.
For $\alpha>0$, we define the functions
\begin{align*}
m^+_{\alpha}(r):=\alpha m^{+}\left(\frac{r}{\alpha}\right),
\end{align*}
which approximate the positive part of $r$ as $\alpha\rightarrow 0$ and $(m_\alpha^+)'(r)\rightarrow {\rm sgn}^+(r)$ as $\alpha\rightarrow 0$. Let $w=\phi(\underline{u_R})-\phi(\overline{u_R})$ and $z=\phi(\overline{v_R})-\phi(\underline{v_R})$.
 Multiplying equation for $u_R$ by $\left(m_{\alpha}^{+}\right)'(w)$ and equation for $v_R$ by $\left(m_{\alpha}^{+}\right)'(z)$. It follows that adding these two equations together that
\begin{align*}
&\int_{0}^{t_0}\int_0^R\left(m_{\alpha}^{+}\right)'(w)(\underline{u_R}-\overline{u_R})_{t}+\left(m_{\alpha}^{+}\right)'(z)(\overline{v_R}-\underline{v_R})_{t}{\rm d}x{\rm d}t\\
\le&-k\int_{0}^{t_0}\int_0^R\left[\left(m_{\alpha}^{+}\right)'(w)-\left(m_{\alpha}^{+}\right)'(z)\right](\underline{u_Rv_R}-\overline{u_Rv_R}){\rm d}x{\rm d}t.
\end{align*}
With the nonlinear function $\phi$, we need to deal with $(m_{\alpha}^{+})'(w)$ and $(m_{\alpha}^{+})'(z)$ to simplify the left hand side.

Now letting $\alpha\rightarrow0$ gives
\begin{align*}
\lim_{\alpha\rightarrow0}\left(m_{\alpha}^+\right)'(w)=\lim_{\alpha\rightarrow0}\left(m_{\alpha}^+\right)'(\phi(\underline{u_R})-\phi(\overline{u_R}))\rightarrow \mbox{sgn}^+(\phi(\underline{u_R})-\phi(\overline{u_R})),
\end{align*}
where $s^+=\max\left\{s,0\right\}$. Note that $\mbox{sgn}^+\left[\phi(\underline{u_R})-\phi(\overline{u_R})\right]=\mbox{sgn}^+(\underline{u_R}-\overline{u_R})$, since $\phi$ is increasing.
By \cite[Lemma 7.6]{epde}, we obtain
\begin{align*}
&\int_{0}^{t_0}\int_0^R\left[\mbox{sgn}^+(\underline{u_R}-\overline{u_R})(\underline{u_R}-\overline{u_R})_t+\mbox{sgn}^+(\overline{v_R}-
\underline{v_R})(\overline{v_R}-\underline{v_R})_t\right]{\rm d}x{\rm d}t\\
=&\int_0^R\left[(\underline{u_R}-\overline{u_R})^++(\overline{v_R}-\underline{v_R})^+\right](x,t_0){\rm d}x-\int_0^R\left[(\underline{u_R}-\overline{u_R})^++(\overline{v_R}-\underline{v_R})^+\right](x,0){\rm d}x\\
\le&-k\int_{0}^{t_0}\int_0^R\left[(\mbox{sgn}w)^+-(\mbox{sgn}z)^+\right](\underline{u_Rv_R}-\overline{u_Rv_R}){\rm d}x{\rm d}t,
\end{align*}
and the expression
\begin{align*}
\left[(\mbox{sgn}w)^+-(\mbox{sgn}z)^+\right](\underline{u_Rv_R}-\overline{u_Rv_R})\ge0.
\end{align*}
Thus
\begin{align}
&\int_0^R\left[(\underline{u_R}-\overline{u_R})^++(\overline{v_R}-\underline{v_R})^+\right](x,t_0){\rm d}x\non\\
\le&-k\int_{0}^{t_0}\int_0^R\left[(\mbox{sgn}w)^+-(\mbox{sgn}z)^+\right](\underline{u_Rv_R}-\overline{u_Rv_R}){\rm d}x{\rm d}t\le0\label{ur5}.    
\end{align}
Hence
\begin{align*}
\left[(\underline{u_R}-\overline{u_R})^++(\overline{v_R}-\underline{v_R})^+\right](\cdot,t_0){\rm d}x=0\quad{\rm on}\ (0,R).
\end{align*}
\end{proof}

The following corollary is immediate from Lemma \ref{lur}.
\begin{corollary}\label{cu1}
Let $\varepsilon>0$. For given initial data $u_{0,R},v_{0,R}$, there is at most one solution $(u_R,v_R)$ of (\ref{f}).
\end{corollary}

If we take $(\underline{u_R},\underline{v_R})=(0,0)$ and $(\overline{u_R},\overline{v_R})=(u_R,v_R)$, then take $(\underline{u_R},\underline{v_R})=(u_R,v_R)$ and $(\overline{u_R},\overline{v_R})=(U_0,V_0$) in Lemma \ref{lur}, we obtain the following.
\begin{corollary}\label{b1}
Let $(u_R,v_R)$ be a weak solution of (\ref{f}). Then we have
\begin{align}
0\le u_R(x,t)\le U_0\quad {\rm and}\quad 0\le v_R(x,t)\le V_0,
\end{align}
for $(x,t)\in(0,R)\times(0,T)$.
\end{corollary}

We will prove the existence of a weak solution of the appropriate problem (\ref{f}) using an iterative method inspired by \cite{nonli}. As the first step in the iteration we consider the problem
\begin{align}
\left\{
\begin{array}{llll}
  \left(u_R^{(1)}\right)_t=\phi\left(u_{R}^{(1)}\right)_{xx}-ku_{R}^{(1)}V_{0},\quad &(x,t)\in(0,R)\times(0,T),\\
  u_{R}^{(1)}(0,t)=U_{0},\quad \phi\left(u_{R}^{(1)}\right)_x(R,t)=0,\quad &\mbox{for}\quad t\in(0,T),\\
  u_{R}^{(1)}(x,0)=u_{0,R}(x),\quad  &\mbox{for}\quad x\in(0,R).
\label{iter1}
\end{array}
\right.
\end{align}
We will prove the existence and uniqueness of a weak solution $u_{R}^{(1)}$ in the following and then substitute $u_{R}^{(1)}$ in the problem
\begin{align}
\left\{
\begin{array}{llll}
  \left(v_{R}^{(1)}\right)_t=\varepsilon\phi(v_{R}^{(1)})_{xx}-ku_{R}^{(1)}v_{R}^{(1)},\quad &(x,t)\in(0,R)\times(0,T),\\
 \phi\left(v_{R}^{(1)}\right)_x(0,t)=\phi\left(v_{R}^{(1)}\right)_x(R,t)=0,\quad &\mbox{for}\quad t\in(0,T),\\
  v_{R}^{(1)}(x,0)=v_{0,R}(x),\quad  &\mbox{for}\quad x\in(0,R),
\label{iter2}
\end{array}
\right.
\end{align}
and obtain a unique weak solution $v_{R}^{(1)}$. Our strategy is to replace $V_0$ in Problem (\ref{iter1}) by $v_{R}^{(1)}$ and again we will have a weak solution $u_{R}^{(2)}$, and so on. In this way, we will obtain sequences $\left\{u_{R}^{(m)}\right\}$ and $\left\{v_{R}^{(m)}\right\}$. Finally letting $m$ tend to infinity, we will obtain a solution of Problem (\ref{f}) in the limit.

In order to be able to carry out this procedure, we first introduce a notion of weak solutions for problems of the following type:
\begin{align}
\left\{
\begin{array}{llll}
  u_{Rt}=\phi(u_R)_{xx}-ku_Rp,\quad &(x,t)\in(0,R)\times(0,T),\\
  u_R(0,t)=U_{0},\quad \phi(u_R)_x(R,t)=0,\quad &\mbox{for}\quad t\in(0,T),\\
  u_R(x,0)=u_{0,R}(x),\quad  &\mbox{for}\quad x\in(0,R),
\label{i}
\end{array}
\right.
\end{align}
and
\begin{align}
\left\{
\begin{array}{llll}
  v_{Rt}=\varepsilon\phi(v_R)_{xx}-kv_Rq,\quad &(x,t)\in(0,R)\times(0,T),\\
  \phi(v_R)_x(0,t)=\phi(v_R)_x(R,t)=0,\quad &\mbox{for}\quad t\in(0,T),\\
  v_R(x,0)=v_{0,R}(x),\quad  &\mbox{for}\quad x\in(0,R),
\label{i2}
\end{array}
\right.
\end{align}
where $0\le p\le V_0$ and $0\le q\le U_0$ almost everywhere in $(0,R)\times(0,T)$.

\begin{definition}
(I). A function $u_R\in L^\infty((0,R)\times(0,T))$ is called a weak solution of problem (\ref{i}) if\\
{\rm (i)} $\phi(u_R)\in\phi(\hat u)+L^2(0,T;\Omega_R);$\\
{\rm (ii)} $u_R$ satisfies
\begin{align*}
\int_0^Ru_{0,R}\xi(x,0){\rm d}x+\int_0^T\int_0^Ru_R\xi_t{\rm d}x{\rm d}t=\int_0^T\int_0^R\phi(u_R)_x\xi_x{\rm d}x{\rm d}t+k\int_0^T\int_0^R\xi  u_Rp{\rm d}x{\rm d}t,
\end{align*}
where $\xi\in \mathcal{F}_T^R$.

(II). A function $v_R\in L^\infty((0,R)\times(0,T))$ is called a weak solution of problem (\ref{i2}) if\\
{\rm (i)} $\phi(v_R)\in L^2(0,T;W^{1,2}(0,R));$\\
{\rm (ii)} $v_R$ satisfies
 \begin{align*}
 \int_0^Rv_{0,R}\xi(x,0){\rm d}x+\int_0^T\int_0^Rv_R\xi_t{\rm d}x{\rm d}t=\int_0^T\int_0^R\varepsilon\phi(v_R)_x\xi_x{\rm d}x{\rm d}t+k\int_0^T\int_0^R\xi v_Rq{\rm d}x{\rm d}t
 \end{align*}
 where $\xi\in \mathcal{F}_T^R$.
\end{definition}

Next we will quote the following lemma which is proved in the Appendix in \cite{nonli}. We will use it to prove the existence of weak solutions of (\ref{i}) and (\ref{i2}).
\begin{lemma}\label{lchi}
Let $\left\{u_{R}^{(n)}\right\}\subset L^\infty((0,R)\times(0,T))$ and $\left\{\phi_n\right\}\subset C(\mathbb{R})$ be sequences with properties
\begin{align*}
&u_{R}^{(n)}\rightharpoonup u_R\quad {\rm in}\ L^2((0,R)\times(0,T)),\\
&\phi_n\quad {\rm is}\ {\rm nondecreasing},\\
&\phi_n\rightarrow\phi\quad {\rm uniformly \ on \ compact \ subset \ of}\ \mathbb{R},\\
&\phi_n(u_{R}^{(n)})\rightarrow\chi\quad {\rm in}\ L^2((0,R)\times(0,T)),
\end{align*}
then $\chi=\phi(u_R)$.
\end{lemma}

Now we can prove the following lemma.
\begin{lemma}\label{lww}
Let $p,q\in L^\infty((0,R)\times(0,T))$ be such that $0\le p\le V_0$, $0\le q\le U_0$. Then problems (\ref{i}) and (\ref{i2}) have unique weak solutions $u_R$ and $v_R$ respectively with the following properties
\begin{align*}
0\le u_R\le U_0,\quad 0\le v_R\le V_0\quad {\rm in}\ (0,R)\times(0,T).
\end{align*}
\end{lemma}
\begin{proof}First we construct the solutions $\left\{u_{R}^{(n)}\right\}$, $\left\{v_{R}^{(n)}\right\}$ of sequences of uniformly parabolic problems in which $\phi$ in (\ref{i}) and (\ref{i2}) have been replaced by smooth functions $\phi_n$ where $\phi_n(U_0)=\phi(U_0)$ and $\phi_n'(u_{R}^{(n)})\ge\frac{1}{n}$.
Under these assumption on $\phi_n$, the equations are parabolic non-degenerate and we may apply standard quasilinear theory to obtain the existence and uniqueness of classical solutions.

Then by the similar arguments to that in \cite[p809]{nonli} we know that $\phi_n\left(u_{R}^{(n)}\right)$ is bounded in ${L^\infty}((0,R)\times(0,T))$, $\phi_n\left(u_{R}^{(n)}\right)_x$ is bounded in ${L^2}((0,R)\times(0,T))$ and $\phi_n\left(u_{R}^{(n)}\right)_t$ is bounded in ${L^2}((0,R)\times(0,T))$.
By \cite[p.170]{measu}, which says that $W^{1,1}(\Omega)\subset BV(\Omega)$, we therefore have $\phi_n\left(u_{R}^{(n)}\right)$ is bounded in $BV((0,R)\times(0,T))$ and there exists a subsequence $\left\{u^{(n_j)}_{R}\right\}$ and a function $\chi_1\in BV((0,R)\times(0,T))$ such that
\begin{align*}
\phi_{n_j}\left(u_{R}^{(n_j)}\right)\rightarrow\chi_1\quad {\rm in}\ L^1((0,R)\times(0,T))\ {\rm as}\ j\rightarrow\infty.
\end{align*}

As $\varepsilon>0$, it follows similarly for $v$ that there exists a subsequence $\left\{v_{R}^{(n_j)}\right\}$ and a function $\chi_2\in BV((0,R)\times(0,T))$ such that
\begin{align*}
\phi_{n_j}\left(v_{R}^{(n_j)}\right)\rightarrow\chi_2\quad {\rm in}\ L^1((0,R)\times(0,T))\ {\rm as}\ j\rightarrow\infty.
\end{align*}

We may choose these sequences such that
\begin{align*}
u_{R}^{(n_j)}\rightharpoonup u_R,\quad v_{R}^{(n_j)}\rightharpoonup v_R \quad {\rm in}\ L^2((0,R)\times(0,T)),
\end{align*}
and the sequence $\left\{\phi_n\right\}$ such that $\phi_n\rightarrow \phi$ uniformly. By Lemma \ref{lchi} we have $\chi_1=\phi(u_R)$, $\chi_2=\phi(v_R)$. We know that $\phi_n\left(u_{R}^{(n)}\right)-\phi(U_0)$ is bounded in $L^2(0,T;\Omega_R)$ and $\phi_n\left(v_{R}^{(n)}\right)$ is bounded in $L^2(0,T;W^{1,2}(0,R))$, so there are subsequences, again denoted by $\left\{u_{R}^{(n_j)}\right\}$ and $\left\{v_{R}^{(n_j)}\right\}$ such that
\begin{align*}
\phi_{n_j}\left(u_{R}^{(n_j)}\right)-\phi(U_0)\rightharpoonup\phi(u_R)-\phi(U_0)\quad{\rm in}\ L^2(0,T;\Omega_R),\\
\phi_{n_j}\left(v_{R}^{(n_j)}\right)\rightharpoonup\phi(v_R)\quad{\rm in}\ L^2((0,R)\times(0,T)).
\end{align*}
By a standard limiting argument we obtain that $u_R$ is a weak solution of Problem (\ref{i}) and $v_R$ is a weak solution of Problem (\ref{i2}).

The uniqueness is shown similarly to the proof of Lemma \ref{lur}.
\end{proof}

From Lemma \ref{lww} we immediately deduce that $u_{R}^{(1)}$ and $v_{R}^{(1)}$ are weak solutions of (\ref{iter1}) and (\ref{iter2}). We then define the sequences $\left\{u_{R}^{(m)}\right\}$ and $\left\{v_{R}^{(m)}\right\}$ inductively as follows, let $u_{R}^{(m)}$ be the weak solution of the problem
\begin{align}
\left\{
\begin{array}{llll}
  u_{Rt}^{(m)}=\phi\left(u_{R}^{(m)}\right)_{xx}-ku_{R}^{(m)}v_{R}^{(m-1)},\quad &(x,t)\in(0,R)\times(0,T),\\
  u_{R}^{(m)}(0,t)=U_{0},\quad \phi\left(u_{R}^{(m)}\right)_x(R,t)=0,\quad &\mbox{for}\quad t\in(0,T),\\
  u_{R}^{(m)}(x,0)=u_{0,R}(x),\quad  &\mbox{for}\quad x\in(0,R),
\label{iter1m}
\end{array}
\right.
\end{align}
let $v_{R}^{(m)}$ be the weak solution of the problem
\begin{align}
\left\{
\begin{array}{llll}
   v_{Rt}^{(m)}=\varepsilon\phi(v_{R}^{(m)})_{xx}-ku_R^{(m)}v_R^{(m)},\quad &(x,t)\in(0,R)\times(0,T),\\
  \phi\left(v_{R}^{(m)}\right)_x(0,t)=\phi\left(v_{R}^{(m)}\right)_x(R,t)=0,\quad &\mbox{for}\quad t\in(0,T),\\
  v_{R}^{(m)}(x,0)=v_{0,R}(x),\quad  &\mbox{for}\quad x\in(0,R).
\label{iter2m}
\end{array}
\right.
\end{align}
Then $u_{R}^{(m)}$ and $v_{R}^{(m)}$ satisfy\\
 {\rm(i)} $\phi\left(u_{R}^{(m)}\right)\in \phi(\hat u)+L^2(0,T;\Omega_R)$,  and
\begin{align*}
&\int_0^Ru_{0,R}\xi(x,0){\rm d}x+\int_0^T\int_0^Ru_{R}^{(m)}\xi_t{\rm d}x{\rm d}t\\=&\int_0^T\int_0^R\phi\left(u_{R}^{(m)}\right)_x\xi_x{\rm d}x{\rm d}t+k\int_0^T\int_0^R\xi u_{R}^{(m)}v_{R}^{(m-1)}{\rm d}x{\rm d}t,
\end{align*}
{\rm(ii)} $\phi\left(v_{R}^{(m)}\right)\in L^2(0,T;W^{1,2}(0,R))$
\begin{align*}
&\int_0^Rv_{0,R}\xi(x,0){\rm d}x+\int_0^T\int_0^Rv_{R}^{(m)}\xi_t{\rm d}x{\rm d}t\\=&\int_0^T\int_0^R\varepsilon\phi\left(v_{R}^{(m)}\right)_x\xi_x{\rm d}x{\rm d}t+k\int_0^T\int_0^R\xi u_{R}^{(m)}v_{R}^{(m)}{\rm d}x{\rm d}t,\end{align*}
where $\xi\in \mathcal{F}_T^R$.

In the following, we prove the monotone dependence of $u_{R}^{(m)},v_{R}^{(m)}$ on $m$.
\begin{lemma}
The problem (\ref{iter1m}) and (\ref{iter2m}) have unique solutions with the following properties
\begin{itemize}\item[{\rm (i)}] $u_{R}^{(m)}$ and $v_{R}^{(m)}$ are weak solutions of problems (\ref{iter1m}) and (\ref{iter2m})$;$
\item[{\rm (ii)}] $0\le u_{R}^{(m)}\le u_{R}^{(m+1)}\le U_0, \qquad 0\le v_{R}^{(m+1)}\le v_{R}^{(m)}\le V_0$.
 \end{itemize}
\label{lwz}
\end{lemma}
\begin{proof}The proof proceeds by induction. We first note that $u_{R}^{(1)}$ and $u_{R}^{(2)}$ satisfy the equations
\begin{align*}
&u_{Rt}^{(1)}=\phi\left(u_{R}^{(1)}\right)_{xx}-ku_{R}^{(1)}V_0,\\
&u_{Rt}^{(2)}=\phi\left(u_{R}^{(2)}\right)_{xx}-ku_{R}^{(2)}v_{R}^{(1)},
\end{align*}
almost everywhere in $(0,R)\times(0,T)$.

Since $u_{R}^{(1)},u_{R}^{(2)}$ satisfy identical initial and boundary conditions, $u_{R}^{(1)}$ is a subsolution for (\ref{iter1m}) with $m=2$, since $v_{R}^{(1)}\le V_0$, which implies $u_{R}^{(2)}\ge u_{R}^{(1)}$. Now consider $v_{R}^{(1)}$ and $v_{R}^{(2)}$
\begin{align*}
v_{Rt}^{(1)}=\varepsilon\phi\left(v_{R}^{(1)}\right)_{xx}-ku_{R}^{(1)}v_{R}^{(1)},\\
v_{Rt}^{(2)}=\varepsilon\phi\left(v_{R}^{(2)}\right)_{xx}-ku_{R}^{(2)}v_{R}^{(2)}.
\end{align*}
Since $v_{R}^{(1)},v_{R}^{(2)}$ satisfy identical initial and boundary conditions, $v_{R}^{(2)}$ is a subsolution for (\ref{iter2m}) as $m=2$, which implies $v_{R}^{(1)}\ge v_{R}^{(2)}$. The proof of monotone dependence of $u_{R}^{(m)}$ and $v_{R}^{(m)}$ of $m$ for large values of $m$ is similar.
\end{proof}

We can now establish the existence of a weak solution of Problem (\ref{f}) when $\varepsilon$ is strictly positive.
\begin{theorem}\label{er}
There exists a unique weak solution $(u_R,v_R)$ of Problem (\ref{f}) such that
\begin{align*}
0\le u_R\le U_0 \quad {\rm and}\quad 0\le v_R\le V_0.
\end{align*}
\end{theorem}
\begin{proof}Lemma \ref{lwz} implies that the functions $u_{R}^{(m)}$ and $v_{R}^{(m)}$ tend (pointwise) to functions $u_R, v_R$ as $m$ tends to infinity.
By the proof of Lemma \ref{lww} we conclude there are subsequences $\left\{u_{R}^{(m_j)}\right\}$ and $\left\{v_{R}^{(m_j)}\right\}$ such that
\begin{align*}
&\phi\left(u_{R}^{(m_j)}\right)-\phi(U_0)\rightharpoonup\phi(u_R)-\phi(U_0)\quad {\rm weakly \ in}\ L^2(0,T;\Omega_R),\\
&\phi\left(v_{R}^{(m_j)}\right)\rightharpoonup\phi(v_R)\quad {\rm weakly \ in}\ L^2(0,T;W^{1,2}(0,R)).
\end{align*}
Then, by the Dominated Convergence Theorem, passing to the limit as $m_j\rightarrow\infty$ leads to
\begin{align*}
\int_0^Ru_{0,R}\xi(x,0){\rm d}x+\int_0^T\int_0^Ru_R\xi_t{\rm d}x{\rm d}t=\int_0^T\int_0^R\phi(u_R)_x\xi_x{\rm d}x{\rm d}t+k\int_0^T\int_0^R\xi u_Rv_R{\rm d}x{\rm d}t,
\end{align*}
where $\xi\in \mathcal{F}_T^R$. We can readily show $u_R$ is a weak solution of (\ref{i}) with $p=v_R$. From DiBenedetto \cite[Theorem 7.1]{dib}, and conclude that $u_R\in C([0,R]\times[0,T])$. Similarly, we know that $v_R$ is a weak solution of problem (\ref{i2}) with $q=u_R$ and we can conclude also that $v_R\in C([0,R]\times[0,T])$. It follows from Lemma \ref{lur} that $(u_R,v_R)$ is the unique weak solution of problem (\ref{f}).
\end{proof}

Next, we will prove the existence of a weak solution of (\ref{a}) with $\varepsilon>0$ by looking at $(u_R,v_R)$ in the limit $R\rightarrow\infty$. First, we prove some preliminary estimates. In the following, $C(L)$ denotes some $L$-dependent constant which varies according to context.
\begin{lemma}\label{lkrb}
Suppose $\varepsilon>0$ and $L>0$. Then there exists a constant $C(L)$ independent of $k$ such that if $R>L+1$, then
\begin{align}
k\int_0^T\int_0^{L+1}u_Rv_R{\rm d}x{\rm d}t\le C(L).
\label{krb}
\end{align}
\end{lemma}
\begin{proof}Introduce a cut-off function $\varphi^1\in C^\infty(\mathbb{R}^+)$ such that $0\le\varphi^1(x)\le 1$ for all $x\in\mathbb{R}^+$, $\varphi^1(0)=\varphi^1_x(0)=0$, $\varphi^{1}(x)=1$ when $x\in[1,2]$ and $\varphi^{1}(x)=0$ when $x\ge3$.
Then given $L\ge2$, define the family of cut-off functions $\varphi^L\in C^\infty(\mathbb{R}^+)$ such that $\varphi^{L}(x)=\varphi^1(x)$ when $x\in[0,1]$, $\varphi^{L}(x)=1$ when $x\in[1,L]$ and $\varphi^{L}(x)=\varphi^{1}(x+2-L)$ when $x\ge L$.
Note that $0\le\varphi^L\le 1$ for all $L$, and $\varphi^L_x$, $\varphi^L_{xx}$ are bounded in $L^\infty(\mathbb{R}^+)$ independently of $L$.

Multiplying the equation for $u_R$ by $\varphi^L$ and integrating over $(0,R)\times(0,T)$ gives that
\begin{align*}
k\int_0^T\int_0^{L+1}u_Rv_R\varphi^L{\rm d}x{\rm d}t=&\int_0^T\int_0^{L+1}\phi(u_R)\varphi^L_{xx}{\rm d}x{\rm d}t+\int_0^{L+1}\varphi^Lu_{0,R}{\rm d}x\\&-\int_0^{L+1}\varphi^Lu_R(x,T){\rm d}x.
\end{align*}
The fact that $0\le u_R\le U_0$, together with the Lebesgue's Monotone Convergence Theorem, yield (\ref{krb}).
\end{proof}

\begin{lemma}\label{lrb}
Suppose $\varepsilon>0$. Then for each $L\ge1$, $\phi(u_R),\phi(v_R)$ are bounded in $L^2\big(0,T;W^{1,2}(0,L)\big)$ independently of $k$ and $R$.
\end{lemma}
\begin{proof}Now we introduce a cut-off function $\psi^{1}\in C^{\infty}\left(\mathbb{R}^{+}\right)$ such that $0\le\psi^{1}\le1$ for $x\in\mathbb{R}^{+}$, $\psi^{1}=1$ when $x\le 1$ and $\psi^{1}=0$ when $x\ge 2$.
Then given $L\ge1$, define the family of cut-off functions $\psi^{L}\in C^{\infty}(\mathbb{R}^{+})$ such that $\psi^{L}=1$ when $x\le L$ and $\psi^{L}=\psi^{1}(x+1-L)$ when $x\ge L$.
Clearly $\psi^{L}$, $\psi^{L}_{x}$ and $\psi^{L}_{xx}$ are bounded in $L^{\infty}(\mathbb{R}^+)$ independently of $L$. Suppose that $R>L+1$. Then multiplying the equation for $u_R$ by $\big[\phi(u_R)-\phi(U_0)\big]\psi^L$ and integrating over $(0,R)\times(0,T)$ gives
\begin{align*}
&\int_0^T\int_0^{L+1}\big[\phi(u_R)-\phi(U_0)\big]\psi^Lu_{Rt}{\rm d}x{\rm d}t\\=&-\int_0^T\int_0^{L+1}|\phi(u_R)_x|^2\psi^L{\rm d}x{\rm d}t+\frac{1}{2}\int_0^T\int_0^{L+1}\big[\phi(u_R)-\phi(U_0)\big]^2\psi^L_{xx}{\rm d}x{\rm d}t\\&-k\int_0^T\int_0^{L+1}\big[\phi(u_R)-\phi(U_0)\big]\psi^L u_Rv_R{\rm d}x{\rm d}t,
\end{align*}
Now letting $F=\displaystyle\int_0^{u_R}\phi(s){\rm d}s\le C$, we have
\begin{align*}
&\int_0^T\int_0^{L+1}|\phi(u_R)_x|^2\psi^L{\rm d}x{\rm d}t\\=&-\int_0^{L+1}\phi(U_0)\big(u_{0,R}-u_R(x,T)\big)\psi^L{\rm d}x-\int_0^{L+1}\big[F(x,T)-F(x,0)\big]\psi^L{\rm d}x\\ &+\frac{1}{2}\int_0^T\int_0^{L+1}\big[\phi(u_R)-\phi(U_0)\big]^2\psi^L_{xx}{\rm d}x{\rm d}t-k\int_0^T\int_0^{L+1}\big[\phi(u_R)-\phi(U_0)\big]\psi^L u_Rv_R{\rm d}x{\rm d}t.
\end{align*}
We know $\phi(u_R)-\phi(U_0)\in L^\infty((0,R)\times(0,T))$, so Lemma \ref{lkrb} yields
\begin{align}
\int_0^T\int_0^{L+1}|\phi(u_R)_x|^2{\rm d}x{\rm d}t\le C,
\label{lrxb1}
\end{align}
independently of $k$ and $R$. If $\varepsilon>0$, the estimate for $\phi(v_R)_x$ can be proved likewise, using the equation for $v_R$.
\end{proof}

In order to prove that the sets $\{u_R\}_{R>0},\{v_R\}_{R>0}$ are each relatively compact in $L^2_{{loc}}(\mathbb{R}^+\times(0,T))$, we now prove estimates of space and time translates of $u_R,v_R$.

It is convenient to introduce a shorthand notation for space and time translates. Given a function $h$, let
\begin{align}
S_\delta h(x,t):=h(x+\delta,t),\quad T_\tau h(x,t):=h(x,t+\tau),
\label{lnst}
\end{align}
for all $(x,t)$ in a suitable space-time domain and appropriate $\delta$ and $\tau$.

As a result of the gradient bounds in Lemma \ref{lrb}, the following result can be proved by adapting the proof of \cite[Lemma 2.6]{spat}.
\begin{lemma}
Suppose $\varepsilon>0$. Then for each $L>0$ and $r\in(0,1)$, there exists a constant $C(L)$, independent of $k$ and $\delta$, such that
\begin{align*}
\int_0^T\int_r^{L+1}|\phi(S_\delta u_R)-\phi(u_R)|^2{\rm d}x{\rm d}t\le C(L)|\delta|^2,\\
\int_0^T\int_r^{L+1}|\phi(S_\delta v_R)-\phi(v_R)|^2{\rm d}x{\rm d}t\le C(L)|\delta|^2,
\end{align*}
for all $\delta\in\mathbb{R}$, $|\delta|\le r$.
\label{lcl}
\end{lemma}

\begin{lemma}\label{ltcl}
Suppose $\varepsilon>0$. Then for each $L>0$, there exists a constant $C(L)$ independent of $k$ and $\tau\in(0,T)$ such that
\begin{align*}
\int_0^{T-\tau}\int_0^{L+1}|\phi(T_\tau u_R)-\phi(u_R)|^2{\rm d}x{\rm d}t\le\tau C(L),\\
\int_0^{T-\tau}\int_0^{L+1}|\phi(T_\tau v_R)-\phi(v_R)|^2{\rm d}x{\rm d}t\le\tau C(L).
\end{align*}
\end{lemma}
\begin{proof}The proof takes the advantage of \cite[Lemma 2.16]{selfsim}, see also \cite[Lemma 3]{spat}. Since we have nonlinear diffusion terms, we also need to deal with the nonlinearity $\phi$. Let $\psi^L$ be as in the proof of Lemma \ref{lrb}. Then it follows using the Mean Value Theorem that
\begin{align*}
&\int_{0}^{T-\tau}\int_0^{L+1}\psi^L\left|\phi(T_\tau u_R)-\phi (u_R)\right|^2{\rm d}x{\rm d}t\\
\le &N \int_{0}^{T-\tau}\int_0^{L+1}\psi^L\left|\phi(T_\tau u_R)-\phi (u_R)\right|\left|T_\tau u_R-u_R\right|{\rm d}x{\rm d}t,
\end{align*}
where $N=\phi'(U_0)$ such that $\phi'(s)\le N$ for all $s\in[0,U_0]$, since $\phi'$ is increasing. Then we have
\begin{align*}
&\int_{0}^{T-\tau}\int_0^{L+1}\psi^L\left|\phi(T_\tau u_R)-\phi (u_R)\right|^2{\rm d}x{\rm d}t\\ \le &N\int_0^\tau\int_{0}^{T-\tau}\int_0^{L+1}\psi^L\left|\phi\left(T_\tau u_R\right)-\phi (u_R)\right|\phi(u_R)_{xx}(x,t+s){\rm d}x{\rm d}t{\rm d}s\\&-Nk\int_0^\tau\int_{0}^{T-\tau}\int_0^{L+1}\psi^L\left|\phi(T_\tau u_R)-\phi (u_R)\right|u_R(x,t+s)v_R(x,t+s){\rm d}x{\rm d}t{\rm d}s\\=&I_1+I_2+I_3,
\end{align*}
with
\begin{align*}
&I_1:=-N\int_0^\tau\int_{0}^{T-\tau}\int_0^{L+1}\psi^L\left|\phi(T_\tau u_R)-\phi (u_R)\right|_x\phi(u_R)_x(x,t+s){\rm d}x{\rm d}t{\rm d}s,\\
&I_2:=-N\int_0^\tau\int_{0}^{T-\tau}\int_0^{L+1}\psi^L_x\left|\phi(T_\tau u_R)-\phi (u_R)\right|\phi(u_R)_x(x,t+s){\rm d}x{\rm d}t{\rm d}s,\\
&I_3:=-Nk\int_0^\tau\int_{0}^{T-\tau}\int_0^{L+1}\psi^L\left|\phi(T_\tau u_R)-\phi (u_R)\right|u_R(x,t+s)v_R(x,t+s){\rm d}x{\rm d}t{\rm d}s.
\end{align*}
$I_1$ can be split into two terms, and using the Cauchy-Schwarz inequality and the property $\psi^L\le 1$ yields
\begin{align*}
|I_1|=&\left|-N\int_0^\tau\int_{0}^{T-\tau}\int_0^{L+1}\psi^L\phi(T_\tau u_R)_x\phi(u_R)_x(x,t+s){\rm d}x{\rm d}t{\rm d}s\right.\\&\left.+N\int_0^\tau\int_{0}^{T-\tau}\int_0^{L+1}\psi^L\phi(u_R)_x\phi(u_R)_x(x,t+s){\rm d}x{\rm d}t{\rm d}s\right|\\
\le&2\tau N\left\{\int_{0}^{T}\int_0^{L+1}\left|\phi(u_R(x,t))_x\right|^2{\rm d}x{\rm d}t\right\}^{\frac{1}{2}}.
\end{align*}
which is bounded by Lemma \ref{lrb}.
By (\ref{b1}) and the Cauchy-Schwarz inequality, there exist $C$ independent of $k$ such that
\begin{align*}
|I_2|\le&\sup|\psi_x^L|NC\tau \left\{\int_{0}^{T-\tau}\int_L^{L+1}\left|\phi(u_R)_x(x,t+s)\right|^2{\rm d}x{\rm d}t\right\}^{\frac{1}{2}}{\rm d}s,
\end{align*}
which is bounded by (\ref{lrxb1}) and the fact that $\sup|\psi_x^L|$ is bounded independent of $L$.
The last term is easier to handle, by (\ref{b1}) we get
\begin{align*}
|I_3|\le 2M\tau N\int_0^T\int_0^{L+1}ku_Rv_R{\rm d}x{\rm d}t,
\end{align*}
which is bounded by Lemma \ref{lkrb}. An analogous estimate for $v_R$ can be obtained by using similar arguments.
\end{proof}

We can now establish the existence of a weak solution of the original problem (\ref{a}) of $S_T$ when $\varepsilon>0$. Now with
\begin{align*}
\Omega_J:=\left\{\alpha\in W^{1,2}((0,J))|\ \alpha=0 \ {\rm at}\ x=0\right\},
\end{align*}
and define
\begin{align*}
\mathcal{F}_T:=\left\{\xi\in C^1(S_T):\ \xi(0,t)=\xi(\cdot,T)=0\  {\rm for}\ t\in(0,T)\ {\rm and}\ {\rm supp}\, \xi\subset[0,J]\times[0,T]\right.\\ \left. {\rm for}\ {\rm some}\ J>0 \right\}.
\end{align*}

\begin{theorem}\label{tq}
Let $\varepsilon>0$. Given $k>0$, there exists a weak solution $(u^k,v^k)\in (L^\infty(S_T))^2$ of (\ref{a}) such that for each $J>0$,
\begin{itemize}
\item[{\rm (i)}]$\phi(u^k)\in\phi(\hat u)+L^2(0,T;\Omega_J),\quad \phi(v^k)\in L^2(0,T;W^{1,2}((0,J)))$
\item[{\rm (ii)}]$(u^k,v^k)$ satisfies
\begin{align*}
&\int_{\mathbb{R}^+}u_{0}^k\xi(x,0){\rm d}x+\iint_{S_T}u^k\xi_{t}{\rm d}x{\rm d}t=\iint_{S_T}\phi(u^k)_x\xi_{x}{\rm d}x{\rm d}t+k\iint_{S_T}\xi u^kv^k{\rm d}x{\rm d}t,\\
&\int_{\mathbb{R}^+}v_{0}^k\xi(x,0){\rm d}x+\iint_{S_T}v^k\xi_{t}{\rm d}x{\rm d}t=\iint_{S_T}\varepsilon\phi(v^k)_x\xi_{x}{\rm d}x{\rm d}t+k\iint_{S_T}\xi u^kv^k{\rm d}x{\rm d}t,
\end{align*}
\end{itemize}
where $\xi\in\mathcal{F}_T$.
\end{theorem}
\begin{proof}Let $u_{0,R},v_{0,R}$ be as in the formulation of problem (\ref{ic1}) and note that as $R\rightarrow\infty$, $u_{0,R}\rightarrow u_0^k$, $v_{0,R}\rightarrow v_0^k$ in $C^1_{loc}(\mathbb{R}^+)$. Then given $R_n\rightarrow\infty$, it follows from the Fr\'echet-Kolmogorov Theorem (see, for example, \cite[Corollary 4.27]{sobs}) and (\ref{b1}), Lemma \ref{lcl} and \ref{ltcl}, that there exist subsequences $\left\{R_{n_j}\right\}$ and functions $u^k\in L^\infty(S_T)$ and $v^k\in L^\infty(S_T)$ such that
\begin{align*}
u_{R_{n_j}}\rightarrow u^k,\quad v_{R_{n_j}}\rightarrow v^k \ \mbox{strongly in}\ L^2_{loc}(S_T)\ \mbox{and a.e. in}\ S_T
\end{align*}
as $j\rightarrow\infty$. Now we know that $\phi(u_{R_{n_j}})-\phi(U_0)$ is bounded in $L^2(0,T;\Omega_J)$ and $\phi(v_{R_{n_j}})$ is bounded in $L^2(0,T;W^{1,2}(0,J))$ by Lemma \ref{lrb}, then we have that, up to a subsequence, as $j\rightarrow\infty$
\begin{align*}
\phi(u_{R_{n_j}})-\phi(U_0)\rightharpoonup\phi(u^k)-\phi(U_0)\quad{\rm in}\ L^2(0,T;\Omega_J),\\
\phi(v_{R_{n_j}})\rightharpoonup\phi(v^k)\quad{\rm in}\ L^2(0,T;W^{1,2}((0,J))).
\end{align*}
By the Dominated Convergence Theorem we can then easily pass to the limit in the weak form of (\ref{a}).
\end{proof}

We use the following comparison principle theorem for (\ref{a}) to show the uniqueness of the weak solution of (\ref{a}). Note that this result covers both the case $\varepsilon>0$ and the case $\varepsilon=0$.

\begin{lemma}\label{lu}
Let $\varepsilon\ge0$ and $(\overline{u},\overline{v})$, $(\underline{u},\underline{v})$ be such that
\begin{itemize}\item[{\rm (a)}] $\overline{u},\underline{u}\in L^\infty(S_T);$
\item[{\rm (b)}]$\phi(\overline{u})\in\phi(\overline{u}(0,\cdot))+L^2(0,T;\Omega_J)$,\ $\phi(\underline{u})\in\phi(\underline{u}(0,\cdot))+L^2(0,T;\Omega_J);$
\item[{\rm (c)}]$\overline{u}_t,\underline{u}_t,\phi(\overline{u})_{xx},\phi(\underline{u})_{xx}\in L^1(S_T);$
\item[{\rm (d)}]$\overline{v},\underline{v}\in L^\infty(S_T)$, $\overline{v}_t,\underline{v}_t\in L^1(S_T);$
\item[{\rm (e)}] If $\varepsilon>0$, $\phi(\overline{v}),\phi(\underline{v})\in L^2(0,T;W^{1,2}((0,J)))$,\
 $\phi(\overline{v})_{xx},\phi(\underline{v})_{xx}\in L^1(S_T);$
 \end{itemize}
and $(\bar{u},\bar{v})$, $(\underline{u},\underline{v})$ satisfy
\begin{align*}
&\overline{u}_{t}\ge\phi(\overline{u})_{xx}-k\overline{uv},\quad\underline{u}_{t}\le\phi(\underline{u})_{xx}-k\underline{u}\underline{v},\quad &&\mbox{in} \ S _{T},\\
&\overline{v}_{t}\le\varepsilon\phi(\overline{v})_{xx}-k\overline{uv},\quad\underline{v}_{t}\ge\varepsilon\phi(\underline{v})_{xx}-k\underline{u}\underline{v},\quad &&\mbox{in} \ S _{T},\\
&\overline{u}(0,\cdot)\ge\underline{u}(0,\cdot),\quad\varepsilon\phi(\overline{v})_x(0,\cdot)=\varepsilon\phi({\underline v})_x(0,\cdot)=0,\quad&&\mbox{on} \ (0,T),\\
&\overline{u}(\cdot,0)\ge\underline{u}(\cdot,0),\quad \overline{v}(0,\cdot)\le\underline{v}(0,\cdot),\quad&&\mbox{on} \ \mathbb{R}^{+}.
\end{align*}
Then
\begin{align*}
\overline{u}\ge\underline{u},\quad\overline{v}\le\underline{v}\quad\mbox{in} \  S_{T}.
\end{align*}
\end{lemma}
\begin{proof}Take the function $m^{+}$ as in the proof of Lemma \ref{lur}, $\psi^L$ be as in the proof of Lemma \ref{lrb} and let $w=\phi(\underline{u})-\phi(\overline{u})$ and $z=\phi(\overline{v})-\phi(\underline{v})$. This result follows from arguments analogous to those used in the proof of Lemma \ref{lur}, replacing $(m^+_\alpha)'(w)$ by $(m^+_\alpha)'(w)\Psi^L$, $(m^+_\alpha)'(z)$ by $(m^+_\alpha)'(z)\Psi^L$ and integrals over $\mathbb{R}^+\times(0,t_0)$ by $(0,R)\times(0,T)$. Letting $\alpha\rightarrow 0$, we have
\begin{align}
&\int_{\mathbb{R}^+}\psi^L\left[(\underline{u}-\overline{u})^++(\overline{v}-\underline{v})^+\right](x,t_0){\rm d}x\non\\
\le&\int_{\mathbb{R}^+}\psi^L\left[(\underline{u}-\overline{u})^++(\overline{v}-\underline{v})^+\right](x,0){\rm d}x+\int_{0}^{t_0}\int_{\mathbb{R}^+}\psi^L_{xx}(w^++\varepsilon z^+){\rm d}x{\rm d}t\non\\
\le&\int_{0}^{t_0}\int_{\mathbb{R}^+}\psi^L_{xx}\left\{\big[\phi(\underline{u})-\phi(\overline{u})\big]^++\varepsilon\big[\phi(\overline{v})-\phi(\underline{v})\big]^+\right\}{\rm d}x{\rm d}t\label{u5},
\end{align}
which is bounded independently of $L$ and $t_0$ by the definition of $\psi^L$ and in particular, $\psi^L_{xx}\neq 0$ only if $x\in[L,L+1]$. Now using Lebesgue's Monotone Convergence Theorem we deduce that
$[(\underline{u}-\overline{u})^++(\overline{v}-\underline{v})^+]\in L^\infty\left(0,T;L^1(\mathbb {R}^+)\right)$, thus (\ref{u5}) tends to $0$ as $L\rightarrow \infty$. Hence
\begin{align*}
\left[(\underline{u}-\overline{u})^++(\overline{v}-\underline{v})^+\right](\cdot,t_0)=0\quad{\rm on}\ \mathbb{R}^+.
\end{align*}
\end{proof}

The following corollaries are immediate from Lemma \ref{lur}.
\begin{corollary}
Let $\varepsilon\ge0$. For given initial data $u_0^k,v_0^k$, there is at most one solution $(u^k,v^k)$ of (\ref{a}).
\label{lu223}
\end{corollary}

\begin{corollary}\label{lbb}
Let $\varepsilon\ge0$ and $(u^k,v^k)$ be a weak solution of (\ref{a}). Then for given $k>0$, we have
\begin{align}
0\le u^k(x,t)\le U_0\quad {\rm and}\quad 0\le v^k(x,t)\le V_0\quad {\rm for}\ (x,t)\in{S_T}.
\label{b}
\end{align}
\end{corollary}

\section{Half-line case: \textit{a priori} bounds, existence and uniqueness of weak solutions for $\varepsilon=0$}
In this section, we prove some \textit{a priori} estimates for $\varepsilon=0$ and for $\varepsilon>0$ that will be used both in proving existence of a weak solution of (\ref{a}) when $\varepsilon=0$ and in the next section, to study the limit of (\ref{a}) as $k\rightarrow\infty$.

The next bound for $ku^kv^k$ is key in the following. The strategy to obtain the estimate is to consider the integral over $(1,\infty)\times(0,T)$ by studying the equation of $u^k$ and the integral over $(0,1)\times(0,T)$ by studying the equation of $v^k$. The proof of this result uses a similar approach to that used in the proof of \cite[Lemma 3.4]{selfsim} and we omit the details. Note that a similar result in the whole-line case is established in Lemma \ref{lk23}, where more involved arguments are needed and we provide a proof.

\begin{lemma}\label{lk}
There exists a constant $C>0$, independent of $\varepsilon\ge0$ and $k>0$, such that for any solution $(u^k,v^k)$ of (\ref{a}), we have
\begin{align*}
\iint_{S_T}ku^k v^k{\rm d}x{\rm d}t\le C.
\end{align*}
\end{lemma}

The following result prove the $L^1$ bounds of $u^k$ and $v^k-V_0$.
\begin{lemma}\label{lkuv}
There exists a constant $C>0$, independent of $\varepsilon\ge 0$ and $k>0$, such that for any solution $(u^k,v^k)$ of (\ref{a}), we have
\begin{align}
\int_{\mathbb{R}^+}u^k(x,t_0){\rm d}x\le C\quad{\rm and}\quad \int_{\mathbb{R}^+}|v^k(x,t_0)-V_0|{\rm d}x\le  C,
\label{lb}
\end{align}
for all $t_0\in[0,T]$.
\end{lemma}

\begin{proof}The estimate for $u^k$ is immediate from the proof of Lemma \ref{lk} and the Monotone Convergence Theorem.
Choose a smooth convex function $m:\mathbb{R}\rightarrow\mathbb{R}$ with
\begin{align*}
m\ge0,\ m(0)=0,\ m'(0)=0,\ m(r)=|r|-\frac{1}{2}\ \mbox{for} \ |r|>1.
\end{align*}
For each $\alpha>0$, define the functions
\begin{align*}
m_{\alpha}(r):=\alpha m(\frac{r}{\alpha}),
\end{align*}
which approximate the modulus function as $\alpha\rightarrow 0$. Denote $\hat v=v^k-V_0$, $\hat z=\phi(v^k)-\phi(V_0)$.

Multiplying the equation of $\hat v$
by $m_\alpha '(\hat z)\psi^L$, where $\psi^L$ as in the proof of Lemma \ref{lrb}, and integrating over $\mathbb{R}^+\times(0,t_0)$, we obtain
\begin{align*}
\int_{0}^{t_0}\int_{\mathbb{R}^+}m_\alpha '(\hat z)\psi^L\hat v_t{\rm d}x{\rm d}t=&\varepsilon\int_{0}^{t_0}\int_{\mathbb{R}^+}m_\alpha '(\hat z)\psi^L\hat{z}_{xx}{\rm d}x{\rm d}t-k\int_{0}^{t_0}\int_{\mathbb{R}^+}m_\alpha '(\hat z)\psi^L u^kv^k{\rm d}x{\rm d}t.
\end{align*}
Letting $\alpha\rightarrow 0$ and \cite[Lemma 7.6]{epde} yields
\begin{align}
&\int_{\mathbb{R}^+}|\hat v(x,t_0)|\psi^L{\rm d}x-\int_{\mathbb{R}^+}|\hat v(x,0)|\psi^L{\rm d}x \non\\ \le&\varepsilon\int_{0}^{t_0}\int_{\mathbb{R}^+}|\hat z|\psi^L_{xx}{\rm d}x{\rm d}t
-k\int_{0}^{t_0}\int_{\mathbb{R}^+}{\rm sgn}(\hat z)\psi^L u^kv^k{\rm d}x{\rm d}t.
\label{lv1}
\end{align}
We know that $\psi^L_{xx}\ne0$ only when $\psi^L_{xx}\in[L,L+1]$ and by Lemma \ref{lk}, the right-hand side of (\ref{lv1}) is bounded independently of $L$ and $k$. So it follows from the fact that $v_0^k-v_0^\infty$ is bounded in $L^1(\mathbb{R^+})$, there exists $C>0$ independent of $k$, such that for all $t_0\in[0,T]$
\begin{align*}
\int_{\mathbb{R}^+}|v^k(x,t_0)-V_0|{\rm d}x\le C.
\end{align*}
\end{proof}

By using the Mean Value Theorem and \textit{a priori} bounds on $u^k,v^k$ on Corollary \ref{lbb}, we can get the following corollary of Lemma \ref{lkuv}.
\begin{corollary}
There exists a constant $C>0$, independent of $\varepsilon\ge0$ and $k>0$, such that for any solution $(u^k,v^k)$ of (\ref{a}), we have
\begin{align}
\int_{\mathbb{R}^+}\phi(u^k)(\cdot,t_0){\rm d}x\le C\quad{\rm and}\quad \int_{\mathbb{R}^+}|\phi(v^k)(\cdot,t_0)-\phi(V_0)|{\rm d}x\le C,
\label{cb}
\end{align}
for all $t_0\in[0,T]$.
\label{lcb}
\end{corollary}

Next we prove a bound for the $L^2$-norm of the space derivatives $\phi(u^k)_x$ and $\phi(v^k)_x$.
\begin{lemma}\label{ldnb}
There exists $C>0$, independent of $\varepsilon\ge0$ and $k>0$, such that for any solution $(u^k,v^k)$ of (\ref{a}),
\label{ln}
\begin{align}
\iint_{S_T}|\phi(u^k)_x|^2{\rm d}x{\rm d}t\le C,\quad{\rm and}\quad \varepsilon\iint_{S_T}|\phi(v^k)_x|^2{\rm d}x{\rm d}t\le C.
\end{align}
\end{lemma}
\begin{proof}The proof follows the similar arguments in Lemma \ref{lrb}. Let $\psi^L$ be as in the proof of Lemma \ref{lrb}. Then multiplying the equation for $u^k$ by $\big[\phi(u^k)-\phi(U_0)\big]\psi^L$ and integrating over $S_T$ give
\begin{align*}
&\iint_{S_T}|\phi(u^k)_x|^2\psi^L{\rm d}x{\rm d}t\\=&-\int_{\mathbb{R}^+}\phi(U_0)\big(u^k_0-u^k(x,T)\big)\psi^L{\rm d}x-\int_{\mathbb{R}^+}\big[F(x,T)-F(x,0)\big]\psi^L{\rm d}x\\ &+\frac{1}{2}\iint_{S_T}\big[\phi(u^k)-\phi(U_0)\big]^2\psi^L_{xx}{\rm d}x{\rm d}t-k\iint_{S_T}\big[\phi(u^k)-\phi(U_0)\big]\psi^L u^kv^k{\rm d}x{\rm d}t,
\end{align*}
where $F=\dint_0^{u^k}\phi(s){\rm d}s$.
By the Mean Value Theorem and (\ref{b}), we obtain for $s\in[0,U_0]$ such that
\begin{align*}
F(x,T)-F(x,0)=\phi(s)u^k(x,T),
\end{align*}
which yields
\begin{align*}
\left|\int_{\mathbb{R}^+}[F(x,T)-F(x,0)]\psi^L{\rm d}x\right|\le\phi(U_0)\displaystyle\sup_{0\le t\le T}\int_{\mathbb{R}^+}u^k(x,t){\rm d}x,
\end{align*}
which is bounded by Lemma \ref{lkuv}. Combining with Lemma \ref{lk}, using Lebesgue's Monotone Convergence Theorem and letting $L\rightarrow\infty$ imply that there exists a constant $C>0$ such that
\begin{align*}
\iint_{S_T}|\phi(u^k)_x|^2{\rm d}x{\rm d}t\le C,
\end{align*}
independently of $k$. If $\varepsilon>0$, the estimate for $\phi(v^k)_x$ can be proved likewise, using the equation for $v^k$.
\end{proof}

The following estimates for the differences of space and time translates of solutions will yield sufficient compactness both to obtain the existence of solutions of (\ref{a}) when $\varepsilon>0$ and $\varepsilon=0$, and to study the strong-interaction limit $k\rightarrow\infty$. The estimates for the differences of space translates of solutions are proved in the similar way to \cite[Lemma 2.15]{selfsim}, which importantly allows $\varepsilon=0$. Note that we need alternative procedures to deal with the nonlinear diffusion, and the monotonicity properties of $\phi$ and \cite[Lemma 7.6]{epde} are both used here.

Recall the notion for space and time translates introduced in (\ref{lnst}).
\begin{lemma}\label{lst}
Suppose that $\varepsilon\ge0$ and let $(u^k,v^k)$ be a solution of (\ref{a}) satisfying (\ref{b}). Then for each $r\in(0,1)$, there exists a function $K_r\ge 0$ independent of $\varepsilon\ge0$ and $k>0$ such that $K_r(\delta)\rightarrow0$ as $|\delta|\rightarrow0$ and for all $|\delta|\le\frac{r}{4}$ and $t\in(0,T)$, we have
\begin{align*}
\int_r^\infty\left|\phi(u^k)-\phi(S_\delta u^k)\right|+\left|\phi(v^k)-\phi(S_\delta v^k)\right|{\rm d}x\le K_r(\delta).
\end{align*}
\end{lemma}
\begin{proof}Let
\begin{align}
&u:=u^k-S_\delta u^k,\quad g:=\phi(u^k)-\phi(S_\delta u^k),\non \\
 &v:=v^k-S_\delta v^k,\quad n:=\phi(v^k)-\phi(S_\delta v^k),
 \label{nuv}
\end{align}
and define a cut-off function $\gamma_r^1\in C^\infty(\mathbb{R}^+)$ such that $0\le\gamma_r^1\le1$, $\gamma_r^1(x)=0$ when $x\in[0,r/2]$, $\gamma_r^1(x)=1$ when $x\in[r,1]$ and $\gamma_r^1(x)=0$ when $x\ge2$.
Then given $L\ge1$ define a family of cut-off function $\gamma_r^L\in C^\infty(\mathbb{R}^+)$ such $\gamma_r^L(x)=\gamma_r^1(x)$ when $x\in[0,r]$, $\gamma_r^L(x)=1$ when $x\in[r,L]$ and $\gamma_r^L(x)=\gamma_r^1(x+1-L)$ when $x\ge L$. Note that $0\le\gamma_r^L\le1$ for all $L$, and $(\gamma_r^L)_x$, $(\gamma_r^L)_{xx}$ are bounded in both $L^\infty(\mathbb{R}^+)$ and $L^1(\mathbb{R^+})$ independently of $L$.

This follows from the similar form of argument used to show in \cite[Lemma 3.7]{selfsim}, but with nonlinear diffusion. We omit most of the details and only note two key calculations involving nonlinear diffusion.

Let $m_\alpha$ be as defined in the proof of Lemma \ref{lkuv}. Letting $\alpha\rightarrow 0$ gives
\begin{align*}
\displaystyle\lim_{\alpha\rightarrow0}(m_\alpha)'(\phi(u^k)-\phi(S_\delta u^k))\rightarrow {\rm sgn}(\phi(u^k)-\phi(S_\delta u^k))={\rm sgn}(u^k-S_\delta u^k).
\end{align*}
Then by \cite[Lemma 7.6]{epde}, we have
\begin{align}
&\int_{\frac{r}{2}}^\infty\gamma_r^L\left\{|u(x,t_0)|+|v(x,t_0)|\right\}{\rm d}x\non\\ \le&\int_{\frac{r}{2}}^\infty\gamma_r^L\left\{|u(x,0)|+|v(x,0)|\right\}{\rm d}x+\int_{0}^{t_0}\int_{\frac{r}{2}}\left(\gamma_r^L\right)_{xx}\left\{|g|+\varepsilon|n|\right\}{\rm d}x{\rm d}t\non\\&-k\int_{0}^{t_0}\int_{\frac{r}{2}}^\infty\gamma_r^L\left[{\rm sgn}(g)+{\rm sgn}(n)\right]\left(u^kv^k-S_\delta u^kS_\delta v^k\right){\rm d}x{\rm d}t\non\\
\le&\int_{\frac{r}{2}}^\infty\gamma_r^L\left\{|u(x,0)|+|v(x,0)|\right\}{\rm d}x+\int_{0}^{t_0}\int_{\frac{r}{2}}^\infty\left(\gamma_r^L\right)_{xx}\left\{|g|+\varepsilon|n|\right\}{\rm d}x{\rm d}t,\label{ls1}
\end{align}
because
\begin{align}
\left[{\rm sgn}(g)+{\rm sgn}(n)\right]\left(u^kv^k-S_\delta u^kS_\delta v^k\right)\ge0.
\end{align}

Now we prove the following bound for the right-hand side of (\ref{ls1}),
\begin{align*}
\int_{0}^{t_0}\int_{\frac{r}{2}}^\infty\left(\gamma_r^L\right)_{xx}|g|{\rm d}x{\rm d}t
\le&|\delta|\left[\int_{0}^{t_0}\int_{\frac{r}{4}}^\infty\left(\gamma_r^L\right)_{xx}^2{\rm d}x{\rm d}t\right]^{\frac{1}{2}}\left[\int_{0}^{t_0}
\int_{\frac{r}{4}}^\infty\left|\phi(u^k)_x(x+\delta,t)\right|^2{\rm d}x{\rm d}t\right]^{\frac{1}{2}}.
\end{align*}
By Lemma \ref{ln} and a similar estimate for $n$, we get
\begin{align}
\int_{0}^{t_0}\int_{\frac{r}{2}}^\infty\left(\gamma_r^L\right)_{xx}\left\{|g|+\varepsilon|n|\right\}{\rm d}x{\rm d}t\le K_r|\delta|,
\label{ls2}
\end{align}
for some constant $K_r$. The result follows from (\ref{ls1}), the fact that $\Vert u^k_0(\cdot+\delta)-u_0^k(\cdot)\Vert_{L^1((r,\infty))}+\Vert v^k_0(\cdot+\delta)-v_0^k(\cdot)\Vert_{L^1((r,\infty))}\le\omega_r(\delta)$ where $\omega_r(\mu)\rightarrow 0$ as $\mu\rightarrow 0$ and Lebesgue's Monotone Convergence Theorem combining with the Mean Value Theorem.
\end{proof}

The estimates for the difference of time translates are proved by using similar methods to those in the proof of Lemma \ref{ltcl}, passing to the limit as $L\rightarrow\infty$ in integrals over $(0,L+1)$ to obtain estimates on integrals over $\mathbb{R}^+$, we leave the details to reader.
\begin{lemma}\label{ltt}
Suppose $\varepsilon\ge0$ and let $(u^k,v^k)$ be a solution of (\ref{a}) satisfying (\ref{b}). Then there exists $C>0$, independent of $\varepsilon$ and $k$, for any $\tau\in(0,T)$ that
\begin{align*}
\int_{0}^{T-\tau}\int_{\mathbb{R}^+}|\phi(T_\tau u^k)-\phi (u^k)|^2{\rm d}x{\rm d}t\le\tau C,\\
\int_{0}^{T-\tau}\int_{\mathbb{R}^+}|\phi(T_\tau v^k)-\phi (v^k)|^2{\rm d}x{\rm d}t\le\tau C.
\end{align*}
\end{lemma}

\begin{lemma}\label{lwc}
Let $(u^k,v^k)$ be weak solutions of (\ref{a}) with $k>0$ and $\varepsilon\ge0$. Then
\begin{align}
\phi(u^k)-\phi(\hat u)\in L^2(0,T;W^{1,2}_0(\mathbb{R}^+)),
\end{align}
and
\begin{align}
\varepsilon[\phi(v^k)-\phi(V_0)]\in L^2(0,T;W^{1,2}(\mathbb{R}^+)),
\end{align}
where $\hat u\in C^\infty(\mathbb{R}^+)$ is a smooth function such that $\hat u=U_0$ when $x=0$ and $\hat u=0$ when $x>1$.
\end{lemma}
\begin{proof}The result for $u^k$ follows from Corollary \ref{lbb} and Corollary \ref{lcb} which ensure that $\phi(u^k)-\phi(U_0)\in L^\infty(S_T)$ and $\phi(u^k)-\phi(U_0)\in L^1(S_T)$, together with Lemma \ref{ldnb} which ensures that $\phi(u^k)_x\in L^2(S_T)$. If $\varepsilon>0$, the estimates for $v^k$ can be proved likewise.
\end{proof}

We can now prove a convergence result for solutions $(u^k,v^k)$ of (\ref{a}) as $\varepsilon\rightarrow 0$.

\begin{lemma}\label{lest}
Let $k>0$ be fixed and $(u_\varepsilon^k,v_\varepsilon^k)$ be solution of (\ref{a}) satisfying (\ref{b}) with $\varepsilon>0$. Then there exist $(u_\star^k,v_\star^k)\in\left(L^\infty(S_T)\right)^2$ such that up to a subsequence, for each $J>0$
\begin{align*}
\phi(u_\varepsilon^k)&\rightarrow \phi(u_\star^k)\quad&&{\rm in}\quad L^2((0,J)\times(0,T)),\\
u_\varepsilon^k&\rightarrow u_\star^k\quad&&{\rm a.e.}\ {\rm in}\quad (0,J)\times(0,T),\\
\phi(v_\varepsilon^k)&\rightarrow \phi(v_\star^k)\quad&&{\rm in}\quad L^2((0,J)\times(0,T)),\\
v_\varepsilon^k&\rightarrow v_\star^k\quad&&{\rm a.e.}\ {\rm in}\quad (0,J)\times(0,T),\\
\phi(u_\varepsilon^k)-\phi(\hat u)&\rightharpoonup \phi(u_\star^k)-\phi(\hat u)\quad&&{\rm in}\quad L^2\left(0,T;W^{1,2}_0(\mathbb{R}^+)\right),
\end{align*}
as $\varepsilon\rightarrow0$, where $\hat u\in C^\infty(\mathbb R^+)$ is a smooth function that $\hat u=U_0$ when $x=0$ and $\hat u=0$ when $x>1$.
\end{lemma}
\begin{proof}It follows from Lemma \ref{lkuv}, Lemma \ref{ldnb} and Corollary \ref{lbb} that $\phi(u_\varepsilon^k)$ and $\phi(v_\varepsilon^k)-\phi(V_0)$ are bounded independently of $\varepsilon\ge0$ in ${L^2(S_T)}$. By Lemma \ref{lst} and Lemma \ref{ltt}, using the Riesz-Fr{\'e}chet-Kolmogorov Theorem \cite[Theorem 4.26]{sobs}, yield that the sets $\left\{\phi(v_\varepsilon^k)-\phi(V_0)\right\}_{\varepsilon>0}$ and $\left\{\phi(u_\varepsilon^k)\right\}_{\varepsilon>0}$ are each relatively compact in $L^2\left((0,J)\times(0,T)\right)$ for each $J>0$. The weak convergence of $\phi(u_\varepsilon^k)-\phi(\hat u)$ in $L^2\left(0,T;W^{1,2}_0(\mathbb{R}^+)\right)$ follows from Lemma \ref{lwc}. Then we know that $\phi(u_\varepsilon^k)\rightarrow \phi(u_\star^k)$ and $\phi(v_\varepsilon^k)\rightarrow \phi(v_\star^k)$ almost everywhere in $(0,J)\times(0,T)$, so since $\phi^{-1}$ is continuous, then we have
$u_\varepsilon^k\rightarrow \phi^{-1}(\phi(u_\star^k))$ and $v_\varepsilon^k\rightarrow \phi^{-1}(\phi(v_\star^k))$ almost everywhere in $(0,J)\times(0,T)$.
\end{proof}

Recall that
\begin{align}
\mathcal{F}_T:=\left\{\xi\in C^1(S_T):\ \xi(0,t)=\xi(\cdot,T)=0\  {\rm for}\ t\in(0,T)\ {\rm and}\ {\rm supp}\,\xi\subset[0,J]\times[0,T]\non \right. \\  \left.{\rm for}\ {\rm some}\ J>0 \right\}.\label{FT}
\end{align}
and
\begin{align*}
\Omega_J:=\left\{\alpha\in W^{1,2}(0,J)|\ \alpha=0 \ {\rm at}\ x=0\right\}.
\end{align*}
Lemma \ref{lu} and Lemma \ref{lest} enable the following result to be established.
\begin{theorem}\label{te0}
Let $\varepsilon=0$ and $k>0$. Then Problem (\ref{a}) has a unique weak solution $(u^k,v^k)\in (L^\infty(S_T))^2$ for each $J>0$ such that
\begin{itemize}

\item[{\rm (i)}]$\phi(u^k)\in\phi(\hat u)+L^2(0,T;\Omega_J)$, where $\hat u\in C^\infty(\mathbb R^+)$ is a smooth function that $\hat u=U_0$ when $x=0$ and $\hat u=0$ when $x>1$;
\item[{\rm (ii)}]$(u^k,v^k)$ satisfies
\begin{align}
&\int_{\mathbb{R}^+}u^k_{0}\xi(x,0){\rm d}x+\iint_{S_T}\xi_t u^k{\rm d}x{\rm d}t=\iint_{S_T}\xi_{x}\phi(u^k)_x{\rm d}x{\rm d}t+k\iint_{S_T}\xi u^kv^k{\rm d}x{\rm d}t,\label{w1}\\
&\int_{\mathbb{R}^+}v^k_0\xi(x,0){\rm d}t+\iint_{S_T}\xi_t v^k{\rm d}x{\rm d}t=k\iint_{S_T}\xi u^kv^k{\rm d}x{\rm d}t,\label{w2}
\end{align}
for all $\xi\in\mathcal{F}_T$.
\end{itemize}
\end{theorem}
\begin{proof}The existence of a solution $(u^k,v^k)$ to (\ref{w1})-(\ref{w2}) follows by using Lemma \ref{lest} to pass to the limit along a subsequence as $\varepsilon\rightarrow0$. The uniqueness of $(u^k,v^k)$ follows from the comparison principle proved in Lemma \ref{lu}.\end{proof}

\section{Half-line case: limit problem for (\ref{a}) as $k\rightarrow\infty$}
We now establish the existence of limits of solutions of (\ref{a}) as $k\rightarrow\infty$, both when $\varepsilon>0$ and $\varepsilon=0$, by using the \textit{a priori} estimates of the previous section.

\begin{lemma}
Let $\varepsilon\ge0$ be fixed and $(u^k,v^k)$ be weak solutions of (\ref{a}) satisfying (\ref{b}) with $k>0$. Then there exists $(u,v)\in(L^\infty(S_T))^2$ such that up to a subsequence, for each $J>0$ that
\begin{align*}
\phi(u^k)&\rightarrow \phi(u)\quad&&{\rm in}\quad L^2((0,J)\times(0,T)),\\
u^k&\rightarrow u\quad&&{\rm a.e.}\ {\rm in}\quad (0,J)\times(0,T),\\
\phi(v^k)&\rightarrow \phi(v)\quad&&{\rm in}\quad L^2((0,J)\times(0,T)),\\
v^k&\rightarrow v\quad&&{\rm a.e.}\ {\rm in}\quad (0,J)\times(0,T),\\
\phi(u^k)-\phi(\hat u)&\rightharpoonup \phi(u)-\phi(\hat u)\quad&&{\rm in}\quad L^2\left(0,T;W^{1,2}_0(\mathbb{R}^+)\right),
\end{align*}
and for $\varepsilon>0$
\begin{align*}
\phi(v^k)-\phi(V_0)&\rightharpoonup \phi(v)-\phi(V_0)\quad&&{\rm in}\quad L^2\left(0,T;W^{1,2}(\mathbb{R}^+)\right),
\end{align*}
as $k\rightarrow\infty$, where $\hat u\in C^\infty(\mathbb R^+)$ is a smooth function that $\hat u=U_0$ when $x=0$ and $\hat u=0$ when $x>1$.
\label{lc}
\end{lemma}
\begin{proof}The proof is directly analogous to that of Lemma \ref{lest}, using bounds independent of $k$ in place of bounds independent of $\varepsilon$. The weak convergence of $\phi(v^k)-\phi(V_0)$ in $L^2\left(0,T;W^{1,2}(\mathbb{R}^+)\right)$ follows from Lemma \ref{lwc}.  \end{proof}

The following segregation result is a key to characterisation of the limits $u,v$ in Lemma \ref{lc}.
\begin{lemma}\label{lkl}
Let $\varepsilon\ge0$ and $(u,v)$ be as in Lemma \ref{lc}. Then
\begin{align}
uv=0\ a.e.\ in\ S_T.
\end{align}

\end{lemma}
\begin{proof}It follows from Lemma \ref{lk} and Lemma \ref{lc} that $uv=0$ almost everywhere in $S_T$ combining with Lemma \ref{lbb} and using  Lebesgue's Dominated Convergence Theorem.
\end{proof}

To derive the limit problem, we set
\begin{align}
w^k:=u^k-v^k,\quad w:=u-v.
\label{d}
\end{align}
Then it follows from Lemma \ref{lc} and Lemma \ref{lkl} that as a sequence $k_n\rightarrow\infty$,
\begin{align*}
w^{k_n}\rightarrow w\quad {\rm in}\ L^2\left((0,J)\times(0,T)\right)\ {\rm for}\ {\rm all}\ J>0\ {\rm and }\ {\rm a.e}\ {\rm in}\ S_T,
\end{align*}
and that
\begin{align*}
u=w^+,\quad v=-w^-,
\end{align*}
where $s^+=\max\left\{0,s\right\}$ and $s^-=\min\left\{0,s\right\}$.

Next result follows using Lemma \ref{lc} and the fact that $u_0^k\rightarrow u^\infty_0$ and $v_0^k\rightarrow v^\infty_0$ in $L^1({\mathbb{R}^+})$ as $k\rightarrow\infty$.
\begin{lemma}
Let $\varepsilon\ge0$ and $(u,v)$ be as in Lemma \ref{lc}. Then
\begin{align}
\iint_{S_T}(u-v)\xi_t{\rm d}x{\rm d}t+\int_{\mathbb{R}^+}(u^\infty_0-v^\infty_0)\xi(x,0){\rm d}x=\iint_{S_T}(\phi(u)-\varepsilon\phi(v))_x\xi_{x}{\rm d}x{\rm d}t,
\label{hll}
\end{align}
for all $\xi\in\mathcal{F}_T$, where $\mathcal{F}_T$ as in (\ref{FT}).
\label{lw2}
\end{lemma}

Now define
\begin{align}
\mathcal{D}(s):=\left\{\begin{aligned}
&\phi(s)\quad &&s\ge0,\\
&-\varepsilon\phi(-s)\quad &&s< 0,\end{aligned}\right.
\label{c0}
\end{align}
and the limit problem
\begin{align}
\left\{\begin{aligned}
&w_t=\mathcal{D}(w)_{xx},\quad &&{\rm in}\ S_T,\\
&w(x,0)=w_0(x):=-V_0, &&{\rm for}\ x>0,\\
&w(0,t)=U_0,&&{\rm for}\ t\in (0,T).
\end{aligned}\right.
\label{c}
\end{align}

\begin{definition}
A function $w$ is a weak solution of (\ref{c}) if
\begin{itemize}
\item[{\rm(i)}]$w\in L^\infty(S_T)$,
\item[{\rm(ii)}]$\mathcal{D}(w)\in\mathcal{D}(\hat w)+L^{2}(0,T;W^{1,2}_0(\mathbb{R}^+))$, where $\hat w\in C^\infty(\mathbb{R}^+)$ is a smooth function with $\hat w=U_0$ when $x=0$ and $\hat w=-V_0$ when $x>1$,
\item[{\rm(iii)}]$w$ satisfies
\begin{align}
\int_{\mathbb{R}^+}w_0(x)\xi(x,0){\rm d}x+\iint_{S_T}w\xi_t{\rm d}x{\rm d}t=\iint_{S_T}\mathcal{D}(w)_x\xi_x{\rm d}x{\rm d}t.
\label{au}
\end{align}
for all $\xi\in\mathcal{F}_T$, where $\mathcal{F}_T$ as in (\ref{FT}).
\end{itemize}
\label{au1}
\end{definition}

\begin{theorem}
Let $\varepsilon\ge0$. The function $w$ defined in (\ref{d}) is a weak solution of problem (\ref{c}) and the whole sequence $(u^k,v^k)$ in Lemma \ref{lc} converges to $(w^+,-w^-)$.
\label{tu}
\end{theorem}

\begin{proof}The existence of a weak solution is a straightforward consequence of Definition \ref{au1} and Lemma \ref{lw2}. The fact that the whole sequence $(u^k,v^k)$ converges to $(w^+,-w^-)$ follows from the uniqueness results proved in Theorem \ref{tu3}.\end{proof}

Now we prove the uniqueness of the weak solution of (\ref{c}). The proof is inspired by \cite[Proposition 5]{wound}, but here we use a different auxiliary function and a different problem to obtain a useful family of test functions.
First, if we choose a smooth test function $\hat\xi\in C^\infty_0(\mathbb{R}^+\times[0,T])$, the weak solution $w$ satisfies
\begin{itemize}
\item[{\rm(i)}]$w\in L^\infty(S_T)$,
\item[{\rm(ii)}]$\mathcal{D}(w)\in\mathcal{D}(\hat w)+L^{2}(0,T;W^{1,2}_0(\mathbb{R}^+))$, where $\hat w\in C^\infty(\mathbb{R}^+)$ is a smooth function with $\hat w=U_0$ when $x=0$ and $\hat w=-V_0$ when $x>1$,
\item[{\rm(iii)}]$w$ satisfies
\begin{align}
\int_{\mathbb{R}^+}w_0(x)\hat\xi(x,0){\rm d}x+\iint_{S_T}w\hat\xi_t{\rm d}x{\rm d}t=\iint_{S_T}\mathcal{D}(w)_x\hat\xi_x{\rm d}x{\rm d}t.
\label{au111}
\end{align}
\end{itemize}

\begin{theorem}\label{tu211}
Let $\varepsilon\ge0$ and consider two solutions $w,\tilde w$ of problem (\ref{c}) with initial data $w_0,\tilde w_0$ respectively, then
\begin{align}
\iint_{S_T}|w-\tilde w|{\rm d}x{\rm d}t\le C(T)\int_{\mathbb{R}^+}|w_0-\tilde w_0|{\rm d}x,
\end{align}
and there exists at most one solution of problem (\ref{c}) for given initial function $w_0$.
\end{theorem}
\begin{proof}We know that there exists $\hat\xi_m\in C^\infty_0(\mathbb{R}^+\times[0,T])$ such that $\hat\xi_m\rightarrow\hat\xi$ in $W^{1,2}_2(S_T)$ as $m\rightarrow\infty$. Now we can rewrite (\ref{au111}) as
\begin{align*}
\int_{\mathbb{R}^+}w_0(x)\hat\xi_m(x,0){\rm d}x+\iint_{S_T}w\hat\xi_{mt}{\rm d}x{\rm d}t=\iint_{S_T}\mathcal{D}(w)\hat\xi_{mxx}{\rm d}x{\rm d}t,
\end{align*}
with $\hat\xi_m\in C^\infty_0(\mathbb{R}\times[0,T))$, then letting $m\rightarrow\infty$, we deduce that for all $\hat\xi\in W^{1,2}_2(S_T)$ with $\hat\xi(\cdot,T)=0$ and $\hat\xi(0,\cdot)=0$, the difference $w-\tilde w$ satisfies
\begin{align}
0=\iint_{S_T}(w-\tilde w)(\hat\xi_t+a\hat\xi_{xx}){\rm d}x{\rm d}t+\int_{\mathbb{R}^+}(w_0-\tilde w_0)\hat\xi(x,0){\rm d}x,\label{tu123}
\end{align}
where $a:=\left\{\begin{aligned}
&\frac{\mathcal{D}(w)-\mathcal{D}(\tilde w)}{w-\tilde w}\quad &&w\ne\tilde w,\\
&0\quad &&{\rm otherwise}.\end{aligned}\right.$

Observe that  $a\in L^\infty(S_T)$, now consider a sequence $\{a_n\}$ of smooth function such that $\frac{1}{n}\le a_n\le\Vert a\Vert_{L^\infty(S_T)}+\frac{1}{n}$ and $\displaystyle\frac{a_n-a}{\sqrt a_n}\rightarrow 0$ almost everywhere in $S_T$ as $n\rightarrow\infty$.

Let $\tilde\xi_n\in W^{1,2}_2(S_T)$ be the solution of problem
\begin{align}
\left\{
\begin{array}{llll}
 \lambda=\tilde\xi_{nt}+a_n\tilde\xi_{nxx},\quad &{\rm in}\ S_T,\\
  \tilde\xi_n(0,t)=0,\quad &t\in(0,T),\\
  \tilde\xi_n(x,T)=0,\quad &x\in\mathbb{R}^{+},
\label{tu12}
\end{array}
\right.
\end{align}
where $\lambda\in C^\infty_c(\mathbb{R}^+\times[0,T))$.

The existence of a solution to this problem follows from standard parabolic theory, see for example \cite[IV, Theorem 9.1]{lady}. We claim that the following estimates hold,
\begin{itemize}
\item[{\rm(i)}]$\Vert\tilde\xi_n\Vert_{L^\infty(S_T)}\le C(\Vert\lambda\Vert_{L^\infty(S_T)},T)$;
\item[{\rm(ii)}]$\displaystyle\iint_{S_T}a_n|\tilde\xi_{nxx}|^2\le C(\lambda,T)$.
\end{itemize}
The maximum principle and a comparison of $\tilde\xi_n$ with the functions $\tilde\xi_n^+,\hat\xi_n^-$ defined by
\begin{align*}
\tilde\xi_n^+=e^{\alpha(T-t)},\quad\tilde\xi_n^-=-e^{\alpha(T-t)},
\end{align*}
where $\alpha=e^{T}\Vert\lambda\Vert_{L^\infty(S_T)}$, gives (i).

To prove (ii), multiplying the equation of $\tilde\xi_n$ by $\tilde\xi_{nxx}$ and integrating over $\mathbb{R}^+\times(t,T)$, we get
\begin{align*}
\int_t^T\int_{\mathbb{R}^+}\lambda\tilde\xi_{nxx}{\rm d}x{\rm d}t=\frac{1}{2}\int_{\mathbb{R}^+}(\tilde\xi_{nx})^2(t){\rm d}x+\int_t^T\int_{\mathbb{R}^+}a_n(\tilde\xi_{nxx})^2{\rm d}x{\rm d}t.
\end{align*}
We deduce from above that
\begin{align*}
\iint_{S_T}a_n(\tilde\xi_{nxx})^2{\rm d}x{\rm d}t\le\Vert\tilde\xi_n\Vert_{L^\infty(S_T)}\Vert\lambda_{xx}\Vert_{L^1(S_T)}.
\end{align*}
Using $\tilde\xi_n$ as a test function of (\ref{tu123}), we obtain
\begin{align*}
0=\iint_{S_T}(w-\tilde w)\left[\lambda+(a-a_n)\tilde\xi_{nxx}\right]{\rm d}x{\rm d}t+\int_{\mathbb{R}^+}(w_0-\tilde w_0)\tilde\xi_n(x,0){\rm d}x.
\end{align*}
We deduce by H\"older's inequality and (ii) that
\begin{align*}
&\lim_{n\rightarrow\infty}\sup\left|\iint_{S_T}(w-\tilde w)(a-a_n)\tilde\xi_{nxx}\right|\\ \le&\lim_{n\rightarrow\infty}\sup\left\Vert\iint_{S_T}(w-\tilde w)\frac{a-a_n}{\sqrt a_n})\right\Vert_{L^2(S_T)}\left\Vert\sqrt a_n\tilde\xi_{nxx}\right\Vert_{L^2(S_T)}=0.
\end{align*}
Thus, in the limit $n\rightarrow\infty$, we get
\begin{align*}
\iint_{S_T}(w-\tilde w)\lambda{\rm d}x{\rm d}t\le C(\Vert\lambda\Vert_{L^\infty(S_T)},T)\int_{\mathbb{R}^+}(w_0-\tilde w_0).
\end{align*}
Taking a sequence $\{\lambda_i\}_{i\in\mathbb{N}}$, $\lambda_i\in C_c^\infty(S_T)$ with $\Vert\lambda_i\Vert_{L^\infty(S_T)}\le2$ and $\lambda_i\rightarrow {\rm sgn}(w-\tilde w)$ almost everywhere, we obtain by letting $i\rightarrow\infty$
\begin{align*}
\iint_{S_T}|w-\tilde w|{\rm d}x{\rm d}t\le C(T)\int_{\mathbb{R}^+}|w_0-\tilde w_0|{\rm d}x.
\end{align*}
\end{proof}
By Theorem \ref{tu} and Theorem \ref{tu211}, we obtain the following.
\begin{theorem}
Let $\varepsilon\ge 0$. Then there exists a unique solution $w$ of the limit problem (\ref{c}).\label{tu3}
\end{theorem}

\section{The whole-line case: Problem (\ref{c})}
In this chapter, we consider the problem (\ref{c}) on the whole real-line by using similar arguments to those used in Section 2 in the half-line case. We omit most of the details and concentrate on differences between half-line case and whole line case.
\subsection{Existence and uniqueness of weak solutions for $\varepsilon>0$}
Let $\varepsilon>0$. Similarly to the half-line case, we use an approximate problem to establish existence of solutions of (\ref{a2}). For each $R>1$, let (\ref{b2}) denote the problem
\begin{align}
\left\{
\begin{array}{llll}
  u_{t}=\phi(u)_{xx}-kuv,\quad &(x,t)\in(-R,R)\times(0,T),\\
  v_{t}=\varepsilon\phi(v)_{xx}-kuv,\quad &(x,t)\in(-R,R)\times(0,T),\\
  \phi(u)_x(-R,t)=\phi(u)_x(R,t)=0,\quad &\mbox{for}\quad t\in(0,T),\\
  \phi(v)_x(-R,t)=\phi(v)_x(R,t)=0,\quad &\mbox{for}\quad t\in(0,T),\\
  u(x,0)=u_{0,R}^k(x),\quad v(x,0)=v_{0,R}^k(x),\quad &\mbox{for}\quad x\in(-R,R),
\label{b2}
\end{array}
\right.
\end{align}
where $u^k_{0,R},v^k_{0,R}\in C^2(\mathbb{R}^+)$ are such that $0\le u^k_{0,R}\le U_0$, $0\le v^k_{0,R}\le V_0$ and
\begin{align}
&u^k_{0,R}=\left\{\begin{aligned}
&U_0-(U_0-u_0^k)\hat\psi^R\ &x<0,\\
  &u_0^k\hat\psi^R\ &x\ge0,\end{aligned}\right.
  \quad v^k_{0,R}=\left\{\begin{aligned}
&v_0^k\hat\psi^R\ &x<0,\\
  &V_0-(V_0-u_0^k)\hat\psi^R\ &x\ge0,\end{aligned}\right.
  \label{ic2}
\end{align}
which define the functions $u^k_{0,R},v^k_{0,R}$ on the whole real line, where the family of cut-off functions $\hat\psi^R\in C^\infty(\mathbb{R}^+)$ with $R>1$ are defined as $\hat\psi^{R}=1$ when $|x|\le R$ and $\hat\psi^{R}=\hat\psi^{1}(x+1-R)$ when $|x|\ge R$
where $\hat\psi^1\in C^{\infty}\left(\mathbb{R}\right)$ is a even, non-negative cut-off function such that $0\le\hat\psi^1(x)\le1$ for all $x\in\mathbb{R}$, $\hat\psi^1(x)=1$ when $|x|\le1$ and $\hat\psi^1(x)=0$ when $|x|\ge2$.

Define the family of test functions
\begin{align}
\mathcal{\hat F}_T:=\left\{\xi\in C^1(Q_T):\ \xi(\cdot,T)=0\ {\rm for}\ t\in(0,T)\ {\rm and}\ {\rm supp}\,\xi\subset[-J,J]\times[0,T]\non \right.\\ \left.{\rm for}\ {\rm some}\ J>0\right\}.\label{F}
\end{align}
Arguments analogous to those used in Section 2 yield existence of solutions of Problem (\ref{a2}) by passing to the limit $R\rightarrow\infty$ in problem (\ref{b2}). We leave the details to reader and simply state the result.
\begin{theorem}
Suppose $\varepsilon>0$. Then for given $k>0$, there exists a weak solution $(u^k,v^k)\in (L^\infty(Q_T))^2$ of (\ref{a2}) such that for each $J>0$
\begin{itemize}
\item[{\rm (i)}]$\phi(u^k)\in L^2(0,T;W^{1,2}((-J,J))),\quad \phi(v^k)\in L^2(0,T;W^{1,2}(-J,J))$;
\item[{\rm (ii)}]$(u^k,v^k)$ satisfies
\begin{align*}
&\int_{\mathbb{R}}u^k_{0}\Psi(x,0){\rm d}x+\iint_{Q_T}u^k\Psi_t{\rm d}x{\rm d}t=\iint_{Q_T}\phi(u^k)_x\Psi_x{\rm d}x{\rm d}t+k\iint_{Q_T}\Psi u^kv^k{\rm d}x{\rm d}t,\\
&\int_{\mathbb{R}}v^k_{0}\Psi(x,0){\rm d}x+\iint_{Q_T}v^k\Psi_t{\rm d}x{\rm d}t=\iint_{Q_T}\varepsilon\phi(v^k)_x\Psi_x{\rm d}x{\rm d}t+k\iint_{Q_T}\Psi u^kv^k{\rm d}x{\rm d}t,
\end{align*}
\end{itemize}
where $\Psi\in\mathcal{\hat F}_T$.
\end{theorem}

The following comparison principle which follows from arguments analogous to those used in the proof of Lemma \ref{lu}, replacing $\psi^L$ by $\hat\psi^L$ and integrals over $S_T$ by integrals over $Q_T$, to prove the uniqueness of the weak solution of Problem ({\ref{a2}}), for both $\varepsilon>0$ and $\varepsilon=0$.
\begin{lemma}\label{lu2}
Suppose $\varepsilon\ge0$ and let $(\overline{u}^k,\overline{v}^k)$, $(\underline{u}^k,\underline{v}^k)$ be such that
\begin{itemize}\item[{\rm (a)}] $\overline{u}^k,\underline{u}^k\in L^{\infty}(Q_T);$
\item[{\rm (b)}]$\phi(\overline{u}^k),\phi(\underline{u}^k)\in L^2(0,T;W^{1,2}(\mathbb{R}))$,\
 $\overline{u}^k_{t},\underline{u}^k_{t},\phi(\overline{u}^k)_{xx},\phi(\underline{u}^k)_{xx}\in L^1(Q_T);$
\item[{\rm (c)}]$\overline{v}^k,\underline{v}^k\in L^{\infty}(Q_T)$, $\overline{v}^k_{t},\underline{v}^k_{t}\in L^1(Q_T);$
\item[{\rm (d)}]If $\varepsilon>0$, $\phi(\overline{v}^k),\phi(\underline{v}^k)\in L^2(0,T;W^{1,2}(\mathbb{R}))$,\
 $\phi(\overline{v}^k)_{xx},\phi(\underline{v}^k)_{xx}\in L^1(Q_T);$
 \end{itemize}
$(\overline{u}^k,\overline{v}^k)$, $(\underline{u}^k,\underline{v}^k)$ satisfy
\begin{align*}
&\overline{u}^k_{t}\ge\phi(\overline{u}^k)_{xx}-k\overline{u}^k\overline{v}^k,\quad\underline{u}^k_{t}\le\phi(\underline{u}^k)_{xx}-k\underline{u}^k\underline{v}^k,&&{\rm in}\ Q_T,\\
&\overline{v}^k_{t}\le\varepsilon\phi(\overline{v}^k)_{xx}-k\overline{u}^k\overline{v}^k,\quad\underline{v}^k_{t}\ge\varepsilon\phi(\underline{v}^k)_{xx}-k\underline{u}^k\underline{v}^k,\quad &&{\rm in}\ Q_T,\\
&\overline{u}^k(\cdot,0)\ge\underline{u}^k(\cdot,0),\quad \overline{v}^k(\cdot,0)\le\underline{v}^k(\cdot,0),\quad&&\mbox{on} \ \mathbb{R}.
\end{align*}
Then
\begin{align*}
\overline{u}^k\ge\underline{u}^k,\quad\overline{v}^k\le\underline{v}^k\quad\mbox{in} \  Q_T.
\end{align*}
\end{lemma}

The following two corollaries follow immediately from Lemma \ref{lu2}.
\begin{corollary}
Suppose $\varepsilon\ge0$ and $k>0$. Then for given initial data $u_0^k,v_0^k$, there is at most one solution $(u^k,v^k)$ of (\ref{a2}).
\end{corollary}

\begin{corollary}\label{ca12}
Let $(u^k,v^k)$ be a weak solution of (\ref{a2}). Then we have
\begin{align}
0\le u^k(x,t)\le U_0\quad {\rm and}\quad 0\le v^k(x,t)\le V_0\quad {\rm for}\ (x,t)\in Q_T.
\label{a12}
\end{align}

\end{corollary}

\subsection{\textit{A priori} bounds and existence of weak solutions for $\varepsilon=0$}
Similar to the half-line case, the following a priori bounds will be used to study the $\varepsilon\rightarrow 0$ and $k\rightarrow\infty$ limits.

The next result follows from a similar argument to that of \cite[Lemma 2.12]{selfsim} and is the whole-line analogue of Lemma \ref{lk}. Since this result is crucial in the following and the whole-line requires the term $kuv$ to be controlled by the equation of $u^k$ on $\mathbb{R}^+$ and the equation of $v^k$ on $\mathbb{R}^-$, we give a short proof here for the convenience of the reader. 
\begin{lemma}\label{lk23}
There exists a constant $C>0$, independent of $\varepsilon\ge0$ and $k>0$, such that for any solution $(u^k,v^k)$ of (\ref{a2}), we have
\begin{align*}
\iint_{Q_T}ku^k v^k{\rm d}x{\rm d}t\le C.
\end{align*}
\end{lemma}
\begin{proof}Define $\beta^{1}\in C^{\infty}\left(\mathbb{R}\right)$ such that $0\le\beta^{1}(x)\le1$ for all $x\in\mathbb{R}$, $\beta^{1}(x)=1$ when $x\in[0,1]$ and $\beta^{1}(x)=0$ when $x\in(-\infty,-1]\cup[2,\infty)$.
Then given $L\ge1$, the family of cut-off functions $\beta^{L}\in C^{\infty}(\mathbb{R})$ are defined by $\beta^{L}(x)=\beta^1$ when $x<0$, $\beta^{L}(x)=1$ when $x\in[0,L]$ and $\beta^{L}(x)=\beta^{1}(x+1-L)$ when $x\ge L$. We also define $\tilde\beta^{L}\in C^{\infty}(\mathbb{R})$ by $\tilde\beta^{L}(x)=\beta^L(-x)$ for all $x\in\mathbb{R}$. Note that $0\le\beta^{L}(x),\tilde\beta^{L}(x)\le1$ for all $x\in\mathbb{R}$ and $\beta^L_x,\beta^L_{xx},\tilde\beta^L_x,\tilde\beta^L_{xx}$ are bounded in both $L^\infty(\mathbb{R})$ and $L^1(\mathbb{R})$ independently of $L$.

Multiplying the equation for $u^k$ by $\beta^L$ and integrating over $\mathbb{R}\times(0,t_0)$ where $t_0\in(0,T]$, give
\begin{align*}
\int_{\mathbb{R}}\beta^L u^k(x,t_0){\rm d}x+k\int_{0}^{t_0}\int_{\mathbb{R}}\beta^Lu^kv^k{\rm d}x{\rm d}t=\int_{0}^{t_0}\int_{\mathbb{R}}\beta^L_{xx}\phi(u^k){\rm d}x{\rm d}t+\int_{\mathbb{R}}\beta^L u^k_0(x){\rm d}x,
\end{align*}
which, by the definition of $\beta^L$, (\ref{a12}) and $u^k_0$ is bounded independently of $k$ in $L^1(\mathbb{R^+})$, imply that the right-hand side is bounded independently of $L$ and $k$, given the existence of $C>0$ such that for all $k>0$ and $t_0\in(0,T]$
\begin{align}
\int_{-1}^{L+1}\beta^L u^k(x,t_0){\rm d}x+k\int_{0}^{t_0}\int_{-1}^{L+1}\beta^Lu^kv^k{\rm d}x{\rm d}t\le C,
\label{bbu}
\end{align}
and then, letting $L\rightarrow\infty$ and using Lebesgue's monotone convergence theorem give
\begin{align}
k\int_{0}^{T}\int_0^\infty u^kv^k{\rm d}x{\rm d}t\le C.
\label{lbk1}
\end{align}

Similarly, since $\{v^k_0\}$ is bounded independently of $k$ in ${L^1(\mathbb{R^-})}$, multiplying the equation for $v^k$  by $\tilde\beta^L$ and integrating over $\mathbb{R}\times(0,t_0)$ yields that $C$ can be chosen large enough that for all $L$ and $k>0$, we have
\begin{align}
\int^1_{-\infty}\tilde\beta^L v^k(x,t_0){\rm d}x+k\int_{0}^{t_0}\int^1_{-\infty}\tilde\beta^Lu^kv^k{\rm d}x{\rm d}t\le C,
\label{bbv}
\end{align}
and hence, letting $L\rightarrow\infty$ yields that
\begin{align}
k\int_{0}^{T}\int^0_{-\infty} u^kv^k{\rm d}x{\rm d}t\le C.
\label{lbk2}
\end{align}
The result then follows from (\ref{lbk1}) and (\ref{lbk2}).\end{proof}

The $L^1$-bounds of $\{u^k(\cdot,t)-u^\infty_0\}$ and $\{v^k(\cdot,t)-v^\infty_0\}$ follow from a similar approach to that used in the proof of \cite[Lemma 2.13]{selfsim}.
\begin{lemma}\label{lb22}
There exists a constant $C>0$ independently of $\varepsilon\ge0$ and $k>0$, such that for any solution $(u^k,v^k)$ of (\ref{a2}), we have
\begin{align}
\Vert u^k(\cdot,t)-u^\infty_0\Vert_{L^1(\mathbb{R})}\le C\ {\rm and}\ \Vert v^k(\cdot,t)-v^\infty_0\Vert_{L^1(\mathbb{R})}\le C.
\end{align}

\end{lemma}

The following corollary is immediate from (\ref{a12}) and the Mean Value Theorem.
\begin{corollary}\label{lcb22}
There exists a constant $C>0$ independently of $\varepsilon\ge0$ and $k>0$, such that for any solution $(u^k,v^k)$ of (\ref{a2}), we have
\begin{align}
\Vert\phi(u^k)(\cdot,t_0)-\phi(u^\infty_0)\Vert_{L^1(\mathbb{R})}\le C\quad{\rm and}\quad \Vert\phi(v^k)(\cdot,t_0)-\phi(v^\infty_0)\Vert_{L^1(\mathbb{R})}\le C,
\label{cb22}
\end{align}
for all $t_0\in[0,T]$.

\end{corollary}

The following is the whole-line analogue of Lemma \ref{ldnb}, which follows from the the similar arguments to the proof of Lemma \ref{ldnb}, replacing $\psi^L$ with $\hat\psi^L$ and integrals over $S_T$ with $Q_T$.
\begin{lemma}\label{ldnb2}
Suppose that $\varepsilon\>0$. Then there exists $C>0$, independent of $\varepsilon>0$ and $k>0$, such that for any solution $(u^k,v^k)$ of (\ref{a2}),
\label{ln2}
\begin{align}
\iint_{Q_T}|\phi(u^k)_x|^2{\rm d}x{\rm d}t\le C,\quad{\rm and}\quad \varepsilon\iint_{Q_T}|\phi(v^k)_x|^2{\rm d}x{\rm d}t\le C.
\end{align}
\end{lemma}

Recall the notion for space and time translates introduced in (\ref{lnst}). The following results are the estimates for the differences of space and time translates of solutions, which follow from the same form of arguments used to show \cite[Lemma 2.15]{selfsim}. Here we using ${\rm sgn}(\phi(u^k)-\phi(S_\delta u^k))={\rm sgn}(u^k-S_\delta u^k)$ and \cite[Lemma 7.6]{epde} to deal with the nonlinear diffusion.
\begin{lemma}\label{lst2}
Suppose $\varepsilon\ge0$ and let $(u^k,v^k)$ be a solution of (\ref{a2}) satisfying (\ref{a12}). Then there exists a function $K\ge 0$ independent of $\varepsilon\ge0$ and $k>0$ such that $K(\delta)\rightarrow0$ as $|\delta|\rightarrow0$, and for $t\in(0,T]$,
\begin{align*}
\int_{\mathbb{R}}\left|\phi(u^k)-\phi(S_\delta u^k)\right|+\left|\phi(v^k)-\phi(S_\delta v^k)\right|{\rm d}x\le K(\delta).
\end{align*}
\end{lemma}

The following result follows from arguments analogous to those used in the proof of Lemma \ref{ltt}, replacing $\psi^L$ by $\hat\psi^L$ and integrals over $\mathbb{R}^+$ by integrals over $\mathbb{R}$.
\begin{lemma}\label{ltt2}
Suppose $\varepsilon\ge0$ and let $(u^k,v^k)$ be a solution of (\ref{a2}) satisfying (\ref{a12}). Then there exists $C>0$, independent of $\varepsilon$ and $k$, such that for any $\tau\in(0,T)$,
\begin{align*}
\int_{0}^{T-\tau}\int_{\mathbb{R}}|\phi(T_\tau u^k)-\phi (u^k)|^2{\rm d}x{\rm d}t\le\tau C,\\
\int_{0}^{T-\tau}\int_{\mathbb{R}}|\phi(T_\tau v^k)-\phi (v^k)|^2{\rm d}x{\rm d}t\le\tau C.
\end{align*}
\end{lemma}

We can now prove a convergence result for solution $(u^k,v^k)$ of (\ref{a2}) as $\varepsilon\rightarrow0$.
\begin{lemma}\label{lest2}
Let $k>0$ be fixed and $(u_\varepsilon^k,v_\varepsilon^k)$ be solution of (\ref{a2}) satisfying (\ref{a12}) with $\varepsilon>0$. Then there exist $(u_\star^k,v_\star^k)\in\left(L^\infty(Q_T)\right)^2$ such that up to a subsequence, for each $J>0$
\begin{align*}
\phi(u_\varepsilon^k)&\rightarrow \phi(u_\star^k)\quad&&{\rm in}\quad L^2((0,J)\times(0,T)),\\
u_\varepsilon^k&\rightarrow u_\star^k\quad&&{\rm a.e.}\ {\rm in}\quad (0,J)\times(0,T),\\
\phi(v_\varepsilon^k)&\rightarrow \phi(v_\star^k)\quad&&{\rm in}\quad L^2((0,J)\times(0,T)),\\
v_\varepsilon^k&\rightarrow v_\star^k\quad&&{\rm a.e.}\ {\rm in}\quad (0,J)\times(0,T),\\
\phi(u_\varepsilon^k)-\phi(\tilde u)&\rightharpoonup \phi(u_\star^k)-\phi(\tilde u)\quad&&{\rm in}\quad L^2\left(0,T;W^{1,2}(\mathbb{R})\right),
\end{align*}
as $\varepsilon\rightarrow0$, where $\tilde u\in C^\infty(\mathbb R)$ is a smooth function such that $\tilde u=u^\infty_0$ for all $|x|>1$.
\end{lemma}
\begin{proof}It follows from Corollary \ref{ca12}, Lemma \ref{lcb22}, \ref{lst2}, \ref{ltt2}, \ref{ldnb2} together with the Riesz-Fr{\'e}chet-Kolmogorov Theorem \cite[Theorem 4.26]{sobs}. We know that $\phi(u_\varepsilon^k)\rightarrow \phi(u_\star^k)$ and $\phi(v_\varepsilon^k)\rightarrow \phi(v_\star^k)$ almost everywhere in $(-J,J)\times(0,T)$, so since $\phi^{-1}$ is continuous, then we have
$u_\varepsilon^k\rightarrow \phi^{-1}(\phi(u_\star^k))$ and $v_\varepsilon^k\rightarrow \phi^{-1}(\phi(v_\star^k))$ almost everywhere in $(-J,J)\times(0,T)$.
\end{proof}

Lemma \ref{lest2} and Corollary \ref{ca12} enable the following result to be established using arguments similar to those that yield Theorem \ref{te0}. We omit details of the proof.

\begin{theorem}
Let $\varepsilon=0$ and $k>0$. Then Problem (\ref{a2}) has a unique weak solution $(u^k,v^k)\in(L^\infty(Q_T))^2$ such that
\begin{itemize}
\item[{\rm (i)}]$\phi(u^k)\in L^2(0,T;W^{1,2}((-J,J)))$;
\item[{\rm (ii)}]$(u^k,v^k)$ satisfies
\begin{align*}
&\int_{\mathbb{R}}u^k_{0}\Psi(x,0){\rm d}x+\iint_{Q_T}u^k\Psi_t{\rm d}x{\rm d}t=\iint_{Q_T}\phi(u^k)_x\Psi_{x}{\rm d}x{\rm d}t+k\iint_{Q_T}\Psi u^kv^k{\rm d}x{\rm d}t,\\
&\int_{\mathbb{R}}v^k_{0}\Psi(x,0){\rm d}x+\iint_{Q_T}v^k\Psi_t{\rm d}x{\rm d}t=k\iint_{Q_T}\Psi u^kv^k{\rm d}x{\rm d}t,
\end{align*}
for all $\Psi\in\mathcal{\hat F}_T$, where $\mathcal{\hat F}_T$ is defined in (\ref{F}).
\end{itemize}
\end{theorem}

\subsection{The limit problem for (\ref{a2}) as $k\rightarrow\infty$}
The next result follows directly from arguments similar to those used in Section 2.3, exploiting the whole-line estimates established in Section 4.2.
\begin{lemma}\label{lli}
Let $\varepsilon\ge0$ be fixed and $(u^k,v^k)$ be solutions of (\ref{a}) satisfying (\ref{a12}) with $k>0$. Then there exists $(u,v)\in(L^\infty(Q_T))^2$ such that up to a subsequence, for each $J>0$
\begin{align*}
\phi(u^k)&\rightarrow \phi(u)\quad&&{\rm in}\quad L^2((-J,J)\times(0,T)),\\
u^k&\rightarrow u\quad&&{\rm a.e.}\ {\rm in}\quad (-J,J)\times(0,T),\\
\phi(v^k)&\rightarrow \phi(v)\quad&&{\rm in}\quad L^2((-J,J)\times(0,T)),\\
v^k&\rightarrow v\quad&&{\rm a.e.}\ {\rm in}\quad (-J,J)\times(0,T),\\
\phi(u^k)-\phi(\tilde u)&\rightharpoonup \phi(u)-\phi(\tilde u)\quad&&{\rm in}\quad L^2\left(0,T;W^{1,2}(\mathbb{R})\right),
\end{align*}
and for $\varepsilon>0$
\begin{align*}
\phi(v^k)-\phi(\tilde v)&\rightharpoonup \phi(v)-\phi(\tilde v)\quad&&{\rm in}\quad L^2\left(0,T;W^{1,2}(\mathbb{R})\right),
\end{align*}
as $k\rightarrow\infty$, where $\tilde u,\tilde v\in C^\infty(\mathbb R)$ are smooth functions such that $\tilde u=u^\infty_0,\tilde v=v_0^\infty$ for all $|x|>1$. Moreover
\begin{align}
uv=0\ {\rm a.e.}\ in\ Q_T.
\end{align}

\end{lemma}
Taking $w^k$ and $w$ as 
\begin{align}
w^k:=u^k-v^k,\quad w:=u-v,\label{d22}
\end{align}
we clearly again have that as a sequence $k_n\rightarrow\infty$, $w^{k_n}\rightarrow w$ in $L^2(Q_T)$ and almost everywhere in $Q_T$, and that $u=w^+,\ v=w^-$.

The following result focus on the function $u-v$, which is useful on the derivation of the limit problem, which follows by using Lemma \ref{lli} and the fact that $u^k_0\rightarrow u^\infty_0$ and $v^k_0\rightarrow v^\infty_0$ as $k\rightarrow\infty$.
\begin{lemma}\label{lw22}
Let $\varepsilon\ge0$ and $(u,v)$ be as in Lemma \ref{lli}. Then
\begin{align*}
\iint_{Q_T}(u-v)\Psi_t{\rm d}x{\rm d}t+\int_{\mathbb{R}}(u^\infty_0-v^\infty_0)\Psi(x,0){\rm d}x=\iint_{Q_T}\left[\phi(u)_x-\varepsilon\phi(v)_x\right]\Psi_{x}{\rm d}x{\rm d}t,
\end{align*}
for all $\Psi\in\mathcal{\hat F}_T$.

\end{lemma}

Now recall the definition of $\mathcal{D}$ from (\ref{c0}) and define the limit problem
\begin{align}
\left\{\begin{aligned}
&w_t=\mathcal{D}(w)_{xx},\quad {\rm in}\ Q_T,\\
&w(x,0)=w_0(x):=\left\{\begin{aligned}&U_0,\quad &&{\rm if}\ x<0,\\-&V_0,&&{\rm if}\ x>0.\end{aligned}\right.
\end{aligned}\right.
\label{c2}
\end{align}

\begin{definition}\label{d122}
A function $w$ is a weak solution of problem (\ref{c2}) if
\begin{itemize}
\item[{\rm(i)}]$w\in L^\infty(Q_T)$,
\item[{\rm(ii)}]$\mathcal{D}(w)\in\mathcal{D}(\hat w)+L^{2}(0,T;W^{1,2}(\mathbb{R}))$, where $\hat w\in C^\infty(\mathbb{R})$ is a smooth function with $\hat w=U_0$ when $x<-1$ and $\hat w=-V_0$ when $x>1$,
\item[{\rm(iii)}]$w$ satisfies for all $T>0$
\begin{align}
\int_{\mathbb{R}}w_0\Psi(x,0){\rm d}x+\iint_{Q_T}w\Psi_t{\rm d}x{\rm d}t=\iint_{Q_T}\mathcal{D}(w)_x\Psi_{x}{\rm d}x{\rm d}t,
\label{wlw}
\end{align}
for all $\Psi\in\mathcal{\hat F}_T$.
\end{itemize}

\end{definition}

\begin{theorem}\label{tu4}
The function $w$ defined in (\ref{d22}) is a unique weak solution of problem (\ref{c2}) and the whole sequence $(u^k,v^k)$ in Lemma \ref{lli} converges to $(w^+,-w^-)$.
\end{theorem}
\begin{proof}The existence of a weak solution is a straightforward consequence of Definition \ref{d122} and Lemma \ref{lw22}. The fact that the whole sequence $(u^k,v^k)$ converges to $(w^+,-w^-)$ follows from the uniqueness, which can be proved by using the arguments analogous to those used in Theorem \ref{tu211}, replacing spatial domain $\mathbb{R}^+$ by $\mathbb{R}$.
\end{proof}

\section*{Acknowledgements}
The authors gratefully acknowledge funding from the EPSRC EP/W522545/1. This paper is based on part of the corresponding author's Ph.D thesis at Swansea University.

\end{document}